\newif\ifdraft
\draftfalse

\documentclass[reqno,a4paper]{amsart}

\usepackage{mathrsfs,amsmath,amssymb,graphicx}
\usepackage{stmaryrd}
\usepackage[unicode]{hyperref}
\usepackage{latexsym}
\usepackage{layout}
\usepackage[english]{babel}
\usepackage{color}
\usepackage{subcaption}
\usepackage{fancyvrb}
\usepackage{color}
\usepackage{amsmath,amscd}
\usepackage{enumitem}

\usepackage{multirow}

\usepackage{mathtools}
\usepackage{tikz} 

\usetikzlibrary{matrix,arrows,decorations,shapes,calc}
\tikzset{every picture/.append style={remember picture},
na/.style={baseline=-.5ex}}

\theoremstyle{plain}
\newtheorem{theorem}{Theorem}[section]
\newtheorem{lemma}[theorem]{Lemma}
\newtheorem{proposition}[theorem]{Proposition}

\newtheorem{corollary}[theorem]{Corollary}

\theoremstyle{definition}
\newtheorem{definition}{Definition}[section]
\theoremstyle{remark}
\newtheorem{remark}{Remark}[section]

\newcommand{\nada}[1]   {}

\newcommand{\pair} {(\mathbb S^3,\mathbf{\Sigma})}
\newcommand{\pairprime} {(\mathbb S^3,\mathbf{\Sigma}')}

\newcommand{\Stwo} {\mathfrak S}
\newcommand{\surface}{\mathbf{\Sigma}}

\newcommand{\Sbb} {\mathbb S}
\newcommand{\SSS} {\mathcal S}
\newcommand{\solidpart} {solid part~}

\newcommand{\sphere} {\Sbb^3}
\newcommand{\sphereminussurface} {\sphere\setminus\surface}

\definecolor{mygray}{rgb}{0.92,0.92,0.92}

\newcommand{\Compl}[1]{\overline{\sphere \setminus #1}}

\newcommand{\op}[1]{\operatorname{#1}}
\newcommand{\sd}{\operatorname{sd}}

\newcommand{\HL}{{\rm HL}}
\newcommand{\HK}{{\rm HK}}

\newcommand{\draftGGG}[1]{\ifdraft{\color{red}#1}\fi}
\newcommand{\draftYYY}[1]{\ifdraft{\color{orange}#1}\fi}

\numberwithin{equation}{section}
\numberwithin{figure}{section}

\title{A complete invariant for closed surfaces in the three-sphere}
\author{Giovanni Bellettini}
\address{Dipartimento di Ingegneria dell'Informazione e Scienze Matematiche, Universit\`a di Siena, 53100 Siena, Italy,
and International Centre for Theoretical Physics ICTP,
Mathematics Section, 34151 Trieste, Italy
}
\email{bellettini@diism.unisi.it}
\author{Maurizio Paolini}
\address{Dipartimento di Matematica e Fisica, Universit\`a Cattolica del Sacro Cuore, 25121 Brescia, Italy}
\email{maurizio.paolini@unicatt.it}
\author{Yi-Sheng Wang}
\address{National Center for Theoretical Sciences, Mathematics Division, Taipei 106, Taiwan}
\email{yisheng@ncts.ntu.edu.tw}

\begin{document}

%{\tiny \tableofcontents}
\thanks{}

\begin{abstract}
%Closed surfaces in $3$-space
%generalize the notion of handlebody links. 
%In this paper, 
Associated to an embedded surface in the $3$-sphere, 
we construct a diagram of fundamental groups,
%, which generalizes the knot group and its peripheral system,  
%associated to a surface embedded in $3$-space,
and prove that it is a complete invariant, 
wherefrom we deduce complete invariants of handlebody links,
tunnels of handlebody links, and spatial graphs.
%of embedding surfaces in 
%the $3$-sphere.  
The main ingredients in the proof of the completeness are  
a generalization of the Kneser conjecture for $3$-manifolds with boundary 
proved also here, and extensions of Waldhausen's theorem by Evans, Tucker and Swarup.    
Computable invariants of handlebody links derived therefrom are calculated.  
\end{abstract}

\maketitle

%%%%%%%%%%%%%%%%%%%%%%%%%%%%%%%%%%%%%%%%%%%%%%%%%%%%%%%%%%%%%%%%%%%%%%%%%%%%%
%
% now only "intro.tex" is included.
%
% for any new section simply add the corresponding two lines here
% and modify accordingly the Makefile
%
%%%%%%%%%%%%%%%%%%%%%%%%%%%%%%%%%%%%%%%%%%%%%%%%%%%%%%%%%%%%%%%%%%%%%%%%%%%%%

%%%%%%%%%%%%%%%%%%%%%%%%%%%%%%%%%%%%%%%%%%%%%%%%%%%%%%%%%%%%%%%%%%%%%%%%%%%%%
\section{Introduction}\label{sec:intro}
 
%The knot group, together with 
%distinguishes all prime knots \cite{Whi:87},  
%its peripheral system, distinguishes all knots,
%up to mirror image \cite{Wal:68ii}, \cite{GorLue:89}.

While a knot, a handlebody knot of genus one, 
is determined by its complement \cite{GorLue:89},
there are infinitely many inequivalent 
handlebody knots of genus $g>1$ with homeomorphic complements 
\cite{Mot:90}, \cite{LeeLee:12}, \cite[Theorem $4.1$]{Suz:75}. 
%hence in general, 
%the embedding type of a closed, 
%\emph{connected} surface $\Sigma$ in $\sphere$
%is not determined by its complements.  
However, it is shown in \cite{BePaWa:19} that, 
up to mirror image, 
the ambient isotopy type of a closed, \emph{connected} surface $\Sigma$ in the 
$3$-sphere $\sphere$ can be recovered from  
the span of fundamental groups
\begin{equation}\label{intro:eq:fund_span}
\pi_1(E)\leftarrow \pi_1(\Sigma)\rightarrow \pi_1(F),
\end{equation}
where $E,F$ are the closures of components of $\sphere\setminus\Sigma$.
In particular, if $E$ is a handlebody, the sequence of 
groups
\begin{equation}\label{intro:eq:complete_invariant_hk}
\op{Ker}\big(\pi_1(\Sigma)\rightarrow \pi_1(E)\big)\leqslant
\pi_1(\Sigma)\rightarrow \pi_1(F)
\end{equation}
determines, up to mirror image, 
the ambient isotopy type of the handlebody knot $E\subset \sphere$.
\eqref{intro:eq:complete_invariant_hk} is reminiscent of 
the knot group with peripheral system \cite{Wal:68ii} 
while \eqref{intro:eq:fund_span} bears 
a resemblance of the fact that, up to homeomorphisms, a Heegaard splitting 
$(V,\Sigma,W)$ of a $3$-manifold is determined by the span
\begin{equation}
\pi_1(V)\leftarrow \pi_1(\Sigma)\rightarrow \pi_1(W).
\end{equation}

%(see also \cite{BePaWa:19}).
%On the other hand, a Heegaard splitting $(\Sigma)$
%of a $3$-manifold 
%$E\subset\sphere$ 
%
%not determined by
%the fundamental group $\pi_1(\sphere\setminus E)$ 
%and its peripheral system, the conjugacy class of the image of 
%$\pi_1(\partial E)\rightarrow \pi_1(\overline{\sphere\setminus E})$, 
%as there are inequivalent handlebody knots 
%with homeomorphic complements \cite{Mot:90}, \cite{LeeLee:12}, \cite{BePaWa:19}.
%If, however, the induced homomorphism
%$\pi_1(\partial E)\rightarrow \pi_1(E)$ 
%and the orientation of $\partial E$ are taken into account, 
%we can obtain a complete invariant.
%It is shown in \cite{BePaWa:19} that the ambient isotopy 
%type of a \emph{connected} closed surface $\Sigma$ in
%the oriented $3$-sphere $\sphere$ with basepoint $\infty$ 
%is determined by the span of fundamental groups
%\begin{equation}\label{intro:eq:fund_span}
%\pi_1(E)\leftarrow \pi_1(\Sigma)\rightarrow \pi_1(F)
%\end{equation}
%
%plus the intersection form on the integral homology $H_1(\Sigma)$, 
%where $E,F$ are the closures of components of $\sphere\setminus \Sigma$
%with $\infty\in F$, and 
%the orientation of $\Sigma$ is induced by $E$. 

The aim of this paper is to generalize 
the invariant \eqref{intro:eq:fund_span}
to closed, {\it not necessarily connected}, surfaces in $\sphere$. 
The situation with non-connected surfaces is more intricate,
and requires  
a generalized Kneser conjecture (Theorem \ref{intro:teo:g_kneser_conjecture}) and extensions 
of the Waldhausen theorem \cite{Eva:72}, \cite{Tuc:73}, \cite{Hem:04}, 
\cite{Swa:80}. A categorical description is also employed 
to simplify the presentation.

Throughout the paper, we work in the piecewise linear category.
%manifolds and maps are assumed to be piecewise linear \cite{Moi:77}.
Unless otherwise specified, $\sphere$ denotes an oriented $3$-sphere
with a base point $\infty$, and $3$-manifolds are assumed to be connected and compact.
Denote by the pair $\SSS=\pair$
a closed surface $\surface$ in $\sphere\setminus \infty$. 
Two pairs $\SSS = (\sphere,\surface)$, $\SSS'=(\sphere,\surface')$
are equivalent if $\surface$, $\surface'$
are isotopic by a basepoint-preserving 
ambient isotopy, or equivalently, if
there exists an orientation-preserving self-homeomorphism
$f$ sending $\surface$ to $\surface'$ with $\infty$ fixed. 
%\cite[Proof of Theorem $3$]{San:60}.
%a reference?
Suppose $\surface$ consists of 
$n$ 
components $\Sigma_i,i=1,\dots, n$.  
Then the closures of 
components
of the complement
$\sphereminussurface$ are
$n+1$ 
%connected, compact 
oriented $3$-manifolds $F_j, j=0,\dots, n$\footnote{The case of a connected $\surface$ corresponds to $n=1$, $F=F_0$, and $E=F_1$.}
(see Fig.\ \ref{fig:pair}).   
By convention, we let $\infty\in F_0$, and  
%$\Sigma_j\subset F_j$ be the surface separating $\infty$ from $F_j$,
%and 
$\Sigma_i$ is so oriented that 
its normal points toward
the side containing $\infty$.
$F_j$ is called a solid part of $\SSS=\pair$.
\begin{figure}[h]
\def\svgwidth{.7\columnwidth}
%% Accompanies image file 'pair_example.pdf' (pdf, eps, ps)
\begingroup%
  \makeatletter%
  \providecommand\color[2][]{%
    \errmessage{(Inkscape) Color is used for the text in Inkscape, but the package 'color.sty' is not loaded}%
    \renewcommand\color[2][]{}%
  }%
  \providecommand\transparent[1]{%
    \errmessage{(Inkscape) Transparency is used (non-zero) for the text in Inkscape, but the package 'transparent.sty' is not loaded}%
    \renewcommand\transparent[1]{}%
  }%
  \providecommand\rotatebox[2]{#2}%
  \newcommand*\fsize{\dimexpr\f@size pt\relax}%
  \newcommand*\lineheight[1]{\fontsize{\fsize}{#1\fsize}\selectfont}%
  \ifx\svgwidth\undefined%
    \setlength{\unitlength}{1388.97637795bp}%
    \ifx\svgscale\undefined%
      \relax%
    \else%
      \setlength{\unitlength}{\unitlength * \real{\svgscale}}%
    \fi%
  \else%
    \setlength{\unitlength}{\svgwidth}%
  \fi%
  \global\let\svgwidth\undefined%
  \global\let\svgscale\undefined%
  \makeatother%
  \begin{picture}(1,0.44897959)%
    \lineheight{1}%
    \setlength\tabcolsep{0pt}%
    \put(0,0){\includegraphics[width=\unitlength,page=1]{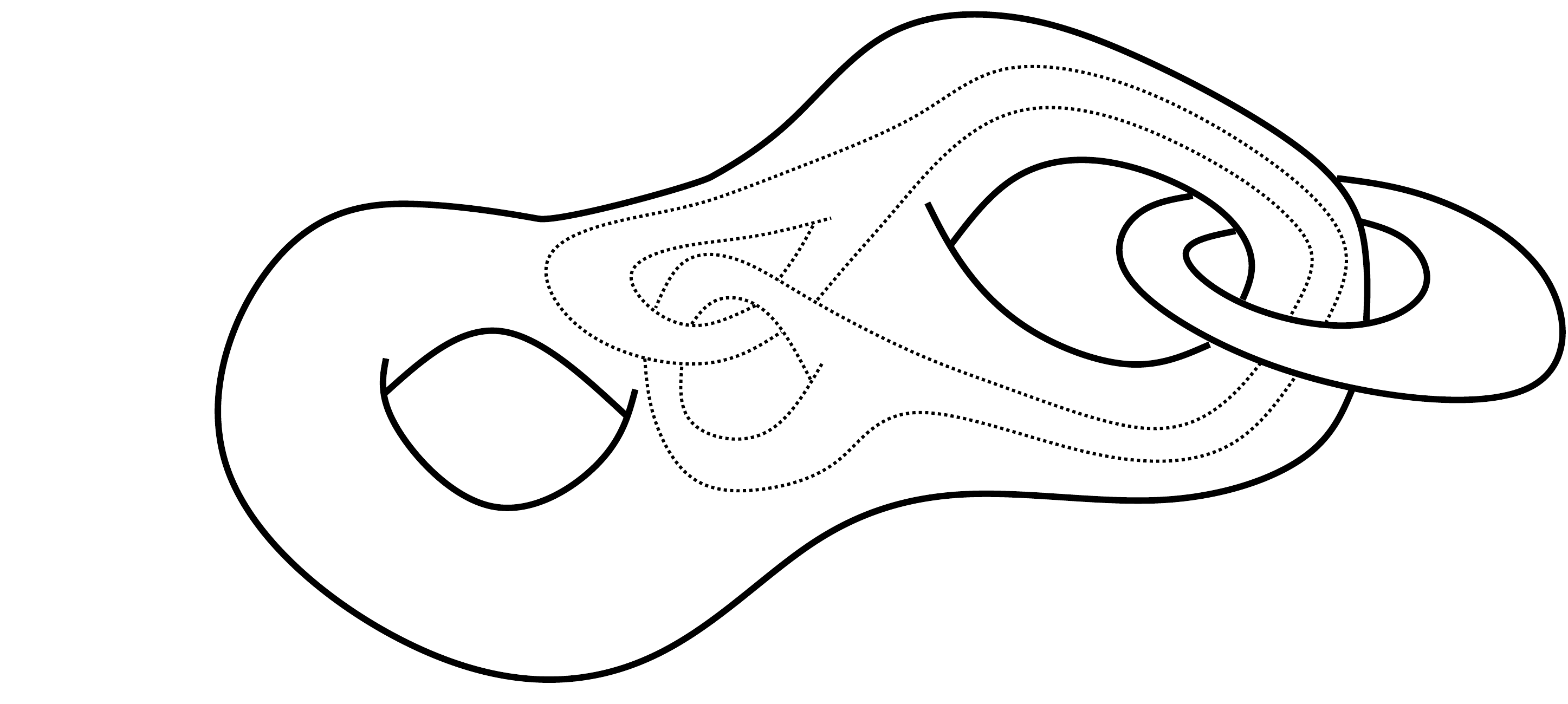}}%
    \put(0.17685758,0.34799459){\color[rgb]{0,0,0}\makebox(0,0)[lt]{\lineheight{1.25}\smash{\begin{tabular}[t]{l}{$F_0$}\end{tabular}}}}%
    \put(0.21933806,0.08158313){\color[rgb]{0,0,0}\makebox(0,0)[lt]{\lineheight{1.25}\smash{\begin{tabular}[t]{l}{$F_2$}\end{tabular}}}}%
    \put(0.89525853,0.21889883){\color[rgb]{0,0,0}\makebox(0,0)[lt]{\lineheight{1.25}\smash{\begin{tabular}[t]{l}{\footnotesize $F_1$}\end{tabular}}}}%
    \put(0.52727861,0.19068447){\color[rgb]{0,0,0}\makebox(0,0)[lt]{\lineheight{1.25}\smash{\begin{tabular}[t]{l}{\tiny $F_3$}\end{tabular}}}}%
    \put(0.0181815,0.03854484){\color[rgb]{0,0,0}\makebox(0,0)[lt]{\lineheight{1.25}\smash{\begin{tabular}[t]{l}$\infty$\end{tabular}}}}%
  \end{picture}%
\endgroup%
  
\caption{An example of $\pair$.}
\label{fig:pair}  
\end{figure}

Each $\Sigma_i$ is the intersection of
exactly two solid parts of $\SSS$, so, thinking of $F_j$ as a node and $\Sigma_i=F_j\cap F_k$
as an edge between the nodes representing $F_j$ and $F_k$, 
we obtain a based tree $\Lambda_\SSS$ with the base node
representing $F_0$ (see \eqref{intro:diag:diag_of_mfds}
for $\Lambda_\SSS$ of Figure \ref{fig:pair}). 
Intuitively, the based tree $\Lambda_\SSS$ measures 
``how far'' from $\infty$ each solid part $F_j$ is, that is, 
how many components of $\surface$ in between $F_j$ and $\infty$. 
Regarding $\Lambda_\SSS$ as an unordered based $1$-dimensional
simplicial complex, we can consider its subdivision
$\sd\Lambda_\SSS$, which comes with a natural partial order on its vertices, 
and hence can be viewed as a based category, a category with a selected base object (e.g. \eqref{intro:diag:diag_of_mfds}).
\begin{equation}\label{intro:diag:diag_of_mfds} 
\begin{tikzpicture}[baseline = (current bounding box.center)]
  \node[draw, circle, fill, label={\tiny $F_1$}, scale=0.3] (F1) at (0,0){};
  \node[draw, star, star points=7, fill, label={\tiny $F_0$}, scale=0.35]  (F0)at (1,0.3){};
  \node[draw, circle, fill, label={\tiny $F_2$}, scale=0.3] (F2)at (2,0) {};
  \node[draw, circle, fill, label={\tiny $F_3$},scale=0.3] (F3)at (3,0) {};
  \draw (F0.center) to node [below]{\tiny $\Sigma_1$}(F1.center);
  \draw (F0.center) to node [below]{\tiny $\Sigma_2$} (F2.center);
  \draw (F2.center) to node [below]{\tiny $\Sigma_3$} (F3.center);
 
  \node[draw, circle, fill, label={\tiny $F_1$}, scale=0.3] (F1) at (5,0){};
  \node[draw, star, star points=7, fill, label={\tiny $F_0$}, scale=0.35]  (F0)at (6,0.3){};
  \node[draw, circle, fill, label={\tiny $F_2$}, scale=0.3] (F2)at (7,0) {};
  \node[draw, circle, fill, label={\tiny $F_3$},scale=0.3] (F3)at (8,0) {};
  
  \node[draw, circle, label=below:{\tiny $\Sigma_1$}, scale=0.2] 
  (S1) at (5.5,0.15){};
  \node[draw, circle, label=below:{\tiny $\Sigma_2$}, scale=0.2]  
  (S2) at (6.5,0.15){};
  \node[draw, circle, label=below:{\tiny $\Sigma_3$}, scale=0.2] 
  (S3)at (7.5,0) {};   
  \draw[->] (S1) to (F0);
  \draw[->] (S1) to (F1);
  \draw[->] (S2) to (F0);
  \draw[->] (S2) to (F2);
  \draw[->] (S3) to (F2);
  \draw[->] (S3) to (F3);
\end{tikzpicture} 
\end{equation}

Taking into account the inclusions 
$\Sigma_i\hookrightarrow F_j$ and $\Sigma_i\hookrightarrow F_k$,
where $\Sigma_i = F_j\cap F_k$, we can think of a pair
$\SSS=\pair$ as a $\sd\Lambda_{\SSS}$-diagram of 
oriented manifolds, that is, a
based functor $\mathcal{MT}(\SSS)$ 
from $\sd\Lambda_\SSS$ to $\mathsf{Mfd}$, the category  
of connected oriented manifolds,
which sends each node $\alpha$ in $\Lambda_\SSS$ to a solid part $F_j$,  
the barycenter $\widehat{\alpha\beta}$ 
of two nodes $\alpha$ and $\beta$ in $\Lambda_{\SSS}$ to the intersection $\Sigma_i=F_j\cap F_k$,   
and each span
\[\alpha\leftarrow \widehat{\alpha\beta}\rightarrow \beta
\quad\text{to}\quad F_j\leftarrow \Sigma_i\rightarrow F_k.\]
%where $F_j$ and $F_k$ are the images of $\alpha$ and $\beta$ under
%$\mathcal{MT}(\SSS)$, respectively.

Now, it is not difficult to see that two pairs $\SSS=\pair$, $\SSS'=\pairprime$ 
are equivalent if and only if the associated 
diagrams of oriented manifolds $\mathcal{MT}(\SSS), \mathcal{MT}(\SSS')$ are equivalent. 
Here by \textit{two $\op{sd}\Lambda$-diagrams of oriented manifolds are equivalent} we understand
there is an equivalence of based categories
\[\mathcal{E}:\sd\Lambda_{\SSS} \mapsto \sd\Lambda_{\SSS'}\]
and a natural transformation $\Phi_M$ 
between $\mathcal{MT}(\SSS)$ and $\mathcal{MT}(\SSS')\circ \mathcal{E}$
such that $\Phi_M(\bullet)$ is an orientation-preserving (o.p.) homeomorphism for
each node $\bullet\in \sd\Lambda_\SSS$:

\begin{center}
\begin{equation}\label{intro:cat_def:equi_two_pairs} 
\begin{tikzpicture}[scale=.8, every node/.style={transform shape}, baseline = (current bounding box.center)]
\node (Lam)  at (0,3) {$\sd\Lambda_{\SSS}$};
\node (Lam') at (0,0) {$\sd\Lambda_{\SSS'}$};
\node (Mfd) at (5, 1.5) {$\mathsf{Mfd}$}; 
\draw[double,->,thick,>=implies] (2,2) to node[right]{\large $\Phi_M$} (2,1);
\draw[->] (Lam)  to node [right]{  $\mathcal{E}$}(Lam');
\draw[->] (Lam') to [out = 0, in =  -135] node [above, xshift =-.5em] {  $\mathcal{MT}(\SSS')$}(Mfd);
\draw[->] (Lam)  to [out = 0, in =  135] node [above] {  $\mathcal{MT}(\SSS)$}  (Mfd); 
\end{tikzpicture}
\end{equation}
\end{center}

\draftGGG{I do not completely understand who is the 
domain of $\Phi$, it is probably my fault. The important
is that this domain contains $\bullet$}
\draftYYY{$\Phi$ is a natural transformation between two functors
$\mathcal{MT}(\SSS)$ and $\mathcal{MT}(\SSS')$,
and $\bullet$ is a node in $\sd\Lambda$. That is, for each node
$\alpha\in \sd\Lambda$, $\Phi(\alpha)$ is a morphism, in this case
a homeomorphism, between the two objects $\mathcal{MT}(\SSS)(\alpha)$ 
and $\mathcal{MT}(\SSS')(\alpha)$, in this case two manifolds.
}
\draftGGG{thanks! do not remove for a while this explanation,
it is useful for me....}
 
Applying the fundamental group functor to $\mathcal{MT}(\SSS)$
and $\mathcal{MT}(\SSS')$, we get two based diagrams of groups, denoted by
\begin{align*}
\mathcal{FT}^u(\SSS):\sd\Lambda_\SSS &\rightarrow \mathsf{Grp}_f\\
\mathcal{FT}^u(\SSS'):\sd\Lambda_{\SSS'} &\rightarrow \mathsf{Grp}_f,
\end{align*} 
where $\mathsf{Grp}_f$ is the category of finitely generated groups
with homomorphisms modulo conjugation.
The question thus arises as to whether an equivalence between 
the induced based diagrams of groups implies an equivalence between $\SSS$ and $\SSS'$. 
Due to the presence of chiral objects, such
as trefoil knots, an affirmative answer is not expected in general. 
However, if we integrate the orientation information in $\SSS=\pair$
into the functor $\mathcal{FT}^u(\SSS)$---the superscript $u$
standing for \textit{unoriented}, then we obtain a complete invariant
of $\SSS$. More precisely, we consider a new functor (fundamental tree)
$\mathcal{FT}(\SSS)$, which is 
the functor $\mathcal{FT}^u(\SSS)$ decorated 
with an intersection form on the abelianization of  
$\mathcal{FT}^u(\SSS)(\widehat{\alpha\beta})$, 
for each barycenter $\widehat{\alpha\beta}$,
namely a non-degenerate bi-linear map on homology groups
\[I: H_{1}(\Sigma)\times H_1(\Sigma)\rightarrow \mathbb{Z},\]
where $\Sigma$ is the component of $\surface$ 
corresponding to
$\widehat{\alpha\beta}$.    
The completeness of the fundamental tree $\mathcal{FT}(\SSS)$
is the main result of this paper. 
\begin{theorem}\label{intro:main_thm}
Two pairs $\SSS, \SSS'$ are equivalent if and only if 
$\mathcal{FT}(\SSS)$ and $\mathcal{FT}(\SSS')$ are
equivalent in the sense that there exist an equivalence
of based categories
\[\mathcal{E} :\sd\Lambda_\SSS\rightarrow \sd\Lambda_{\SSS'}\]
and a natural isomorphism  
\[\Phi: \mathcal{FT}(\SSS)\mapsto \mathcal{FT}(\SSS')\circ \mathcal{E}\]
such that, for each barycenter $\widehat{\alpha\beta}$ in $\sd\Lambda_\SSS$, 
the isomorphism on homology induced by 
$\Phi(\widehat{\alpha\beta})$ preserves intersection forms.  
\end{theorem}

A handlebody link $\HL$ 
is a disjoint, finite union of handlebodies 
$\{F_1,\cdots,F_n\}$ in $\sphere\setminus\infty$. 
The boundary of $F_i$, $i=1,\cdots,n$, namely 
$\surface=\{\Sigma_i:=\partial F_i\}_{i=1}^n$, 
induces a pair $\SSS=\pair$, and two handlebody links 
are ambient isotopic
if and only if the induced pairs are equivalent.   
Theorem \ref{intro:main_thm} immediately implies the
following complete invariant for handlebody links.

\begin{corollary}\label{intro:invariant_handlebody_links}
Let $\SSS=\pair$, $\SSS'=\pairprime$ be pairs induced by
two handlebody links $\HL,\HL'$ of $n$ components. 
Then $\HL$ and $\HL'$ are equivalent 
if and only if there exists
a permutation $\sigma$ on $\{1,\dots, n\}$
and isomorphisms
$\phi_0$ and $\phi_{i}$, $i=1,\cdots,n$, 
such that the diagram
\begin{center} 
\begin{tikzpicture}[scale=.8, every node/.style={transform shape}, baseline = (current bounding box.center)]
\node (F) at (0,1.5) {$\pi_1(F_0)$};
\node (F') at (4,1.5) {$\pi_1(F_0')$};
\node (bF) at (0,0) {$\pi_1(\Sigma_i)$};
\node (bF') at (4,0) {$\pi_1(\Sigma_{\sigma(i)}')$};

\draw[->] (F) to node [above]{\footnotesize $\phi_{0}$}(F');
\draw[->] (bF) to node [right]{}  (F);
\draw[->] (bF')to node [right]{}(F');
\draw[->] (bF) to node [above]{\footnotesize $\phi_{i}$}(bF');
\end{tikzpicture}
\end{center}
commutes up to conjugation, the subgroups
\[\phi_{i}\big(\operatorname{Ker}(\pi_1(\Sigma_i)\xrightarrow{\iota_{i}} \pi_1(F_i))\big),\quad \text{and} \quad
\operatorname{Ker}(\pi_1(\Sigma_{\sigma(i)}')\xrightarrow{\iota_{\sigma(i)}'} \pi_1(F_{\sigma(i)}'))\]
are conjugate in $\pi_1(\Sigma_{\sigma(i)}')$,
and the induced isomorphism on homology by $\phi_i$
preserves intersection forms. 
\end{corollary}  
The proof of Theorem \ref{intro:main_thm} actually
constructs homeomorphisms from $\Sigma_i$ to $\Sigma'_{\sigma(i)}$
realizing $\phi_i$, $i=1,\cdots, n$.

A system of arcs of a handlebody link $\HL$
is a set of disjoint, properly embedded arcs $\mathcal{A}=\{\alpha_1,\cdots,\alpha_k\}$
in $\Compl\HL$. We denote by $\HL^\mathcal{A}$ the 
induced handlebody link $\HL\cup_i\mathfrak{N}(\alpha_i)$,
where $\mathfrak{N}(\alpha_i)$ is a regular neighborhood 
of $\alpha_i$ in $\Compl\HL$, and denote by $D_i^\mathcal{A}\subset \mathfrak{N}(\alpha_i)$ a disk dual to $\alpha_i$. 
%in $\mathfrak{N}(\alpha_i)\subset\HL^\mathcal{A}$. 
Note that $\partial D_i^\mathcal{A}$ determines
an element in $\pi_1(\Compl{\HL^\mathcal{A}})$, up to conjugation and inverse.
The induced set of conjugacy classes 
is denoted by $[\partial D_i^\mathcal{A}]^{\pm 1}$.
Let $\{F_i^\mathcal{A}\}$ be components of $\HL^\mathcal{A}$
and $(\sphere,\surface^\mathcal{A})$ the induced pair.
%with $\surface^\mathcal{A}=\partial \HL^\mathcal{A}$.
%
%
Two systems of arcs 
\[\mathcal{A}=\{\alpha_1,\cdots,\alpha_k\}\quad\text{and}\quad
\mathcal{B}=\{\beta_1,\cdots,\beta_k\}\]
are equivalent if there is an
orientation-preserving 
self-homeomorphism of $\sphere$ 
preserving $\HL$ and sending $\mathcal{A}$ to $\mathcal{B}$ (compare with knot tunnels \cite{BoiRosZie:88}, \cite{Kob:99}).
Theorem \ref{intro:main_thm} and Corollary \ref{intro:invariant_handlebody_links} 
imply a complete invariant
of systems of arcs of $\HL$.
\begin{corollary}\label{intro:invariant_arc_system}
Let 
\[\mathcal{A}=\{\alpha_1,\cdots,\alpha_k\}\quad\text{and}\quad
\mathcal{B}=\{\beta_1,\cdots,\beta_k\}\]
be two systems of arcs of $\HL$.
Then $\mathcal{A},\mathcal{B}$ are equivalent if and only if 
$\HL^\mathcal{A},\HL^\mathcal{B}$
have the same number, say $n$, of components, and
there exist permutations $\sigma$ on $\{1,\cdots,n\}$
and $\tau$ on $\{1,\cdots, k\}$ 
and isomorphisms
$\phi_0$ and $\phi_{i}$, $i=1,\cdots,n$, 
such that the diagram
\begin{center} 
\begin{tikzpicture}[scale=.8, every node/.style={transform shape}, baseline = (current bounding box.center)]
\node (F) at (0,1.5) {$\pi_1(F_0^\mathcal{A})$};
\node (F') at (4,1.5) {$\pi_1(F_0^\mathcal{B})$};
\node (bF) at (0,0) {$\pi_1(\Sigma^\mathcal{A}_i)$};
\node (bF') at (4,0) {$\pi_1(\Sigma^\mathcal{B}_{\sigma(i)})$};

\draw[->] (F) to node [above]{\footnotesize $\phi_{0}$}(F');
\draw[->] (bF) to node [right]{}  (F);
\draw[->] (bF')to node [right]{}(F');
\draw[->] (bF) to node [above]{\footnotesize $\phi_i$}(bF');
\end{tikzpicture}
\end{center} 
commutes up to conjugation, and the subgroups
\begin{equation}\label{eq:subgroups_conjugate}
\phi_{i}\big(\operatorname{Ker}(\pi_1(\Sigma_i^{\mathcal{A}})
%
%\xrightarrow{\iota_{i}} 
%
\rightarrow
\pi_1(F_i^\mathcal{A}))\big)\quad \text{and} \quad
\operatorname{Ker}(\pi_1(\Sigma_{\sigma(i)}^{\mathcal{B}})
%\xrightarrow{\iota_{\sigma(i)}^\mathcal{B}} 
\rightarrow
\pi_1(F_{\sigma(i)}^\mathcal{B}))
\end{equation}
are conjugate in $\pi_1(\Sigma_{\sigma(i)}^\mathcal{B})$
with $\phi_i$ sending $[\partial D_j^\mathcal{A}]^{\pm 1}$
to $[\partial D_{\tau(j)}^\mathcal{B}]^{\pm 1}$ 
whenever $D_j^\mathcal{A}\subset F_i^\mathcal{A}$,
%for every $i,j$,
%where $D_{\tau(j)}^\mathcal{B}\subset F_{\sigma(i)}^\mathcal{B}$,  
and the isomorphism on homology induced by $\phi_i$
preserves intersection forms. 
%where $D_j^\mathcal{A},D_{\tau(j)}^\mathcal{B}$
%are dual disks of $\alpha_j, \beta_{\tau(j)}$, respectively. 
\end{corollary}  
If $\HL$ is a finite union of $3$-balls in $\sphere$,
Corollary \ref{intro:invariant_arc_system} gives 
a complete invariant of spatial graphs, up to ambient isotopy \cite{Tan:94}.
On the other hand, if $\mathcal{A},\mathcal{B}$
are tunnels, that is $\partial \HL^\mathcal{A}\subset\sphere$,
$\partial \HL^\mathcal{B}\subset \sphere$
induce Heegaard splittings of $\sphere$ \cite{EudOza:14}, \cite{Mur:19}, 
Corollary \ref{intro:invariant_arc_system} gives 
a complete invariant of systems of tunnels of $\HL$.

%In the case when HL=Unknotted circle and K--->tunnel of knot K in a solid torus
%system of non-connecting tunnels

%%%%
%%%%
%%%%%
Theorem \ref{intro:main_thm} also readily implies a complete invariant
for unbased pairs $\pair_u$, which are 
surfaces in an \textit{unbased} $\sphere$. 
Two unbased pairs $\pair_u$, $\pairprime_u$ 
are equivalent if there exists an o.p.\ 
self-homeomorphism of $\sphere$ sending $\surface$ to $\surface'$.  
 
\begin{corollary}
The unbased pairs $\pair_u$ and $\pairprime_u$
are equivalent if and only if there exist points 
$\ast\in \sphere\setminus \surface$ and $\ast'\in \sphere\setminus \surface'$
such that $\mathcal{FT}(\SSS)$ and $\mathcal{FT}(\SSS')$ are equivalent, where
$\SSS$ and $\SSS'$ are ``based" pairs
induced by $\ast$ and $\ast'$, respectively.
\end{corollary}

%%%%%%%%%%%%%%%%%%%%%%%%%%%%the below example might not be interesting enough
%%%%%%%%%%%%%%%%%%%%%%%%%%%%to put in the introduction.
\nada{
There are inequivalent pairs that are 
equivalent as unbased pairs. For instance,
the boundary of a toric shell $T \subset \mathbb{R}^3\subset \sphere$
and the boundary of a tubular neighborhood of a Hopf link
$H\subset \mathbb{R}^3\subset \sphere$
are equivalent as embeddings in an unbased $\sphere$
but not in a based $\sphere$.
}

One key ingredient of the proof of Theorem \ref{intro:main_thm}  
is a generalized Kneser conjecture (Theorem  \ref{intro:teo:g_kneser_conjecture}). Here we adopt the convention
that, given a connected subspace $X$ of a space $Y$, 
the notation $\pi_1(X)\xrightarrow{\gamma} \pi_1(Y)$
denotes the homomorphism induced by an arc $\gamma$ 
connecting the base point of $X$
to the base point of $Y$. When there is no need to specify 
the arc, it is dropped from the notation.
\begin{theorem}\label{intro:teo:g_kneser_conjecture}
Let $M$ be a $3$-manifold 
%without 
%spherical boundary components, 
%and $\surface_1$ and $\surface_2$ be disjoint surfaces 
%with $\surface_1\cup\surface_2=\partial M$.
and $\surface_1$, $\surface_2$  
be disjoint closed surfaces on $\partial M$ 
with $\surface_1\cup \surface_2$ containing all 
non-spherical components of $\partial M$. 
Suppose there exists an isomorphism 
\[\pi_1(M)\xrightarrow{\phi} A_1\ast A_2,\]
where $A_1\ast A_2$ is the free product of two groups 
$A_1$, $A_2$, 
such that the composition
\begin{equation}\label{eq:homo_bdry_to_M}
\pi_1(\Sigma_{ij})\xrightarrow{\delta_{ij}}  
\pi_1(M)\xrightarrow{\phi} A_1\ast A_2
\end{equation} 
factors through $A_i\hookrightarrow A_1\ast A_2$
by a homomorphism $\phi_{ij}:\pi_1(\Sigma_{ij})\rightarrow A_i$,
for every component $\Sigma_{ij}$ in $\surface_i$.
Then $M$ is a connected sum   
\[ 
M\simeq M_1\# M_2  \quad with  \quad\partial M_i=\surface_i\cup \{\text{some $2$-spheres}\},
\] 
and there exist $\delta_i,\epsilon_{ij}$, and an isomorphism $\phi_i:\pi_1(M_i)\rightarrow A_i$, $i=1,2$,
such that  
the diagram   
\begin{equation}\label{diag:diag_in_g_Kneser_lemma}
\begin{tikzpicture}[scale=.8, every node/.style={transform shape}, baseline = (current bounding box.center)]
\node (M) at (0,1.5) {$\pi_1(M)$};
\node (Mi) at (0,0) {$\pi_1(M_i)$};
\node (A)  at (4,1.5) {$A_1\ast A_2$};
\node (Ai) at (4,0) {$A_i$};
\node (S) at (0,-1.5) {$\pi_1(\Sigma_{ij})$};

\draw [->] (M) to node [below]{$\sim$} node [above]{$\phi$}(A);
\draw [->] (Mi) to node [below]{$\sim$} node [above]{$\phi_i$}(Ai);
\draw [->] (Mi) to node [right]{$\delta_i$}(M);
\draw [->] (Ai) to (A);
\draw [->] (S) to  node [right]{$\epsilon_{ij}$}(Mi);
\draw [->] (S) to [out=160, in=200] node [left]{$\delta_{ij}$}(M);
\draw [->] (S) to [out=0,in=-90] node [above]{$\phi_{ij}$}(Ai);
\end{tikzpicture} 
\end{equation}
commutes, for every $j$.
\end{theorem} 
The Kneser conjecture, as stated in \cite{Hei:71}, 
\cite[Theorem $7.1$]{Hem:04},
asserts that any 
$\partial$-irreducible 
$3$-manifold $M$ having a free product fundamental group 
is a connected sum that respects the free product. 
In general, $\partial$-irreducibility
cannot be dropped (see \cite[p.974]{Jac:69}).
Theorem \ref{intro:teo:g_kneser_conjecture}  
replaces $\partial$-irreducibility with a
mild algebraic condition, which is, in fact, also a necessary condition, 
and it holds automatically in the $\partial$-irreducible case (see Remark \ref{rmk:classic_Kneser_conjecture}). 
%Note that the condition ``without spherical
%component can be dropped''
%The generalization is
%needed in the proof of Theorem \ref{intro:main_thm}.
 
%%%%%%%%%%%%
%%%%%%%%%%%%%%%%%%%%%%%%%%%%%
%%%%%%%%%%%%%%%%%%%%%%%%%%%%%

The paper is organized as follows: Basic definitions and conventions
are given in Section \ref{sec:preli}. 
Section \ref{sec:graft_fund} discusses 
geometric and algebraic graft decompositions which
allow us to reduce a pair $\SSS=\pair$ to ``irreducible'' ones.  
Section \ref{sec:threemfd}
proves Theorem \ref{intro:teo:g_kneser_conjecture}, 
the generalized Kneser
conjecture, and Section \ref{sec:proof}
is occupied by the proof of 
Theorem \ref{intro:main_thm}. 
Examples are computed in Section \ref{sec:examples}.

\section*{Acknowledgments}
The first author acknowledges from the support of INDAM/GNAMPA, and ICTP 
International Centre for Theoretical Physics (Trieste),
and the third author is supported by National Center for Theoretical Sciences.
%We thank the anonymous referee for helpful comments that simplify the proof
%of the main theorem significantly.

\section{Preliminaries}\label{sec:preli}
%%setup
%%preparation (splittable; non-splittable)
%%Waldhausen's lemma + A variant of Heil's lemma
%%maps between irre. 3-mfds <-> group homomorphisms between their fund. groups.
%%Waldhausen's theorems
 
Here we review some definitions introduced in the introduction
and fix some convention. 
We denote by $\simeq$ an equivalence of pairs, 
and by $\mathsf{Sur}$ 
the set of equivalence classes of pairs.
When there is no danger of confusion, 
$\simeq$ is also used to denote 
other equivalences, such as 
homeomorphisms, isomorphisms, etc.
Bold font $\mathbf{X}$ is reserved for possibly non-connected spaces,
while normal font $X$ stands for a component of $\mathbf{X}$.

Recall that, given a pair $\SSS=\pair$ with
$\surface=\{\Sigma_1,\cdots, \Sigma_n\}$ and
solid parts $\{\infty\in F_0,\cdots,F_n\}$,
the based tree $\Lambda_\SSS$ is given by 
thinking of each $F_j$, $j=0,\dots,n$,
as nodes and $\Sigma_i=F_j\cap F_k$ as edges connecting
nodes representing $F_j$ and $F_k$.

\begin{definition}[\textbf{Depth tree}]
The based tree $\Lambda_\SSS$ 
is called the depth tree of $\SSS=\pair$, and 
a \solidpart $F_j$ is said to have depth $k$ 
if the node in $\Lambda_\SSS$
representing $F_j$ is connected to the base node by $k$ edges; in
particular, $F_0$ has depth $0$.
\end{definition}

\begin{definition}[\textbf{Barycentric diagram}]
Given a based finite graph $\Gamma$, 
the associated barycentric diagram 
$\operatorname{sd}\Gamma$ is a based diagram obtained by
replacing each edge $j-k$ by the span $j\leftarrow \hat{jk}\rightarrow k$. 
\end{definition}

In other words, $\operatorname{sd}\Gamma$ is 
the barycentric subdivision of
$\Gamma$ with $\hat{jk}$ being the barycenter of the edge $j-k$.
Being a diagram, $\operatorname{sd}\Gamma$ can also be viewed as a small category.

\begin{definition}[\textbf{Equivalence}]
Two based finite graphs $\Gamma$ and $\Gamma'$ are equivalent if
there exists an equivalence of based categories
\[\mathcal{E}:\sd\Gamma\rightarrow \sd\Gamma'.\] 
\end{definition}
The above definition is equivalent to saying that $\Gamma$ and $\Gamma'$
are isomorphic as based graphs; for our purpose however, 
it is more
convenient to adopt the categorical definition.
%, which allows
%us to translate the geometric description of $\pair$ into a more 
%categorical one. 
Note that equivalent pairs have equivalent 
depth trees. 
The notion of depth tree comes in handy 
when we discuss the graft decomposition of $\pair$. 
%in Section \ref{sec:graft_fund}.

\begin{definition}[\textbf{Barycentric diagram in a category}]\label{def:bary_diag_in_C}
Given a based finite graph $\Gamma$ and a category $\mathsf{C}$,
a based $\sd \Gamma$-diagram in $\mathsf{C}$, or 
a based barycentric diagram in $\mathsf{C}$ 
of type $\sd \Gamma$,
is a functor $\mathcal{F}$ from $\sd \Gamma$ to $\mathsf{C}$. 
\end{definition}

\begin{definition}[\textbf{Equivalence}]
Let $\mathcal{F}$ be a based $\sd \Gamma$-diagram in $\mathsf{C}$
and $\mathcal{F}'$ 
a $($based$)$ $\sd \Gamma'$-diagram in $\mathsf{C}$.
Then $\mathcal{F}$ and  
$\mathcal{F}'$ 
are equivalent if there exists 
an equivalence of based categories
\[\mathcal{E}:\sd \Gamma\rightarrow \sd\Gamma'\]
and a natural isomorphism 
\[\Phi:\mathcal{F}\Rightarrow \mathcal{F}'\circ \mathcal{E}.\]  
The set of equivalence classes of barycentric diagrams in $\mathsf{C}$ 
is denoted by $\mathsf{C}^{\mathtt{BD}}$.
\end{definition}

The $\sd\Lambda_{\SSS}$-diagram 
$\mathcal{MT}(\SSS)$ associated to a pair $\SSS=\pair$
in the category $\mathsf{Mfd}$ of oriented compact manifolds
and continuous maps is an example of barycentric diagram.  
As explained in the introduction, 
the functor $\mathcal{MT}(\bullet)$ induces an injective mapping
\[\mathcal{MT}:\mathsf{Sur}\rightarrow \mathsf{Mfd}^{\mathtt{BD}}.\]  
%\end{lemma} 
Composing $\mathcal{MT}$ with the fundamental group functor $\pi_1(\bullet)$
gives a $\sd \Lambda_{\SSS}$-diagram $\mathcal{FT}^{u}(\bullet)$ 
in the category of finitely-generated groups $\mathsf{Grp}_f$ 
with homomorphisms modulo conjugation. 
The orientation information is lost in the passage,
and the induced mapping 
\[\mathcal{FT}^u:\mathsf{Sur}\rightarrow \mathsf{Grp}^{\mathtt{BD}}_f\] 
is no longer injective.

\section{Graft decomposition}\label{sec:graft_fund}
%%%$S$ for S2
%%%Use lower script only

\subsection{Geometric graft decomposition}
\begin{definition}[\textbf{Non-splitting sphere}]\label{def:splitting_sphere}
Given a pair $\pair$, a $2$-sphere $\Stwo$ 
with $\Stwo\cap \big(\surface\cup\{\infty\}\big)=\emptyset$,
is non-splitting with respect
to (w.r.t.) $\pair$ if there exists a $3$-ball $B$ bounded by $\Stwo$
with $B\cap \surface=\emptyset$, and is splitting
w.r.t.\ $\pair$ otherwise. 
\end{definition}
 
Note that if $\Stwo$ is in a solid 
part $F$ of $\pair$, then the $3$-ball $B$ 
must also be in $F$. 
If $\infty\in F$, we isotopy 
$B$ in $F$ such that
%Up to ambient isotopy, we may assume 
$\infty\notin B$.

\begin{definition}[\textbf{Splittable/non-splittable pair}]\label{def:split_pair}
A pair $\pair$ is splittable if it admits a splitting sphere; 
otherwise, it is non-splittable.
\end{definition}

The following gluing operation produces split pairs.
\begin{definition}[\textbf{Grafting}]
By grafting a pair $\SSS'=\pairprime$ onto another pair $\SSS=\pair$ at
a \solidpart $F_i$ of $\pair$ we understand 
gluing $\sphere\setminus B'$ and $\sphere\setminus B$
via an orientation-reversing homeomorphism,
where $B'\subset F_0'\setminus \{\infty\}$
and $B\subset F_i$,
%performing the gluing operation between the solid part $F_0'$ of $\pairprime$ that contains the base point and the solid part $F_i$ of $\pair$,
with the base point being the one of $\SSS$.  
The resulting pair is denoted by
\[
\pair\overset{F_i}{\dashleftarrow}\pairprime \qquad \text{or simply} \qquad \SSS\overset{F_i}{\dashleftarrow} \SSS'.
\] 
A pair $\SSS=\pair$ is said to be obtained by performing grafting operations 
finitely many times if $\SSS$ is equivalent to   
\begin{equation}\label{eq:gra_decom}  
\SSS_1
\overset{F_{i_1}^{(1)}}{\dashleftarrow}\SSS_2
\overset{F_{i_2}^{(2)}}{\dashleftarrow}\cdots\overset{F_{i_{k-1}}^{(k-1)}}{\dashleftarrow}\SSS_k,
\end{equation} 
\[\hspace*{-4em}\text{where $F_{i_j}^{(j)}$ is a solid part of  } \SSS_1\overset{F_{i_1}^{(1)}}{\dashleftarrow}\SSS_2
\overset{F_{i_2}^{(2)}}{\dashleftarrow}\cdots\overset{F_{i_{j-1}}^{(j-1)}}{\dashleftarrow}\SSS_j, j=2,\dots, k-1.
\]  
\eqref{eq:gra_decom} is called a graft decomposition of $\SSS$
of length $k$. 
%is the length of the graft decomposition.
If $\SSS_j$ is non-splittable for each $j$,
%, $j=1,\dots, k$, 
then \eqref{eq:gra_decom} is a non-splittable graft decomposition of $\SSS$.
\end{definition}
It is not difficult to see that the grafting operation
depends only on $\pair,\pairprime$, and the solid part of $\pair$ 
to be grafted on. We drop $F_{i_j}^{(j)}$ in \eqref{eq:gra_decom} from the notation when on which solid parts it is grafted is irrelevant.

\begin{proposition}[\textbf{Non-splittable graft
decomposition}]\label{prop:nonsplit_graft_decomp}
Every pair $\SSS$ admits a non-splittable graft decomposition. 
Furthermore, if  
\begin{equation}\label{eq:non_split_graft_decomps}
\SSS_1\dashleftarrow \SSS_2\dashleftarrow\dots
\dashleftarrow \SSS_m \quad {and} \quad 
\SSS_1'\dashleftarrow \SSS_2'\dashleftarrow\dots
\dashleftarrow  \SSS_p'  
\end{equation} 
are two non-splittable graft decompositions of $\SSS$, 
then $m=p$, and after reindexing if necessary, 
$\SSS_i$ and $\SSS_i'$ are equivalent, for every $i$.
\end{proposition}
\begin{proof}
First we index the solid parts $F_j$, $j=0,\dots,n$, of
$\SSS=\pair$ in such a way that $j>i$ if 
$F_j$ has a greater depth than $F_i$.
We shall prove the existence and uniqueness by induction on $n$.
 
\textbf{Existence:} 
Consider the set 
\begin{equation}\label{eq:nonirre_comp_of_greatest_depth}
\{j\mid F_j \text{ is not irreducible }\}.
\end{equation}
If the set \eqref{eq:nonirre_comp_of_greatest_depth} is empty, 
for instance, when $n=0$ and $n=1$, then 
$\SSS$ is non-splittable, and there is nothing to prove. 

Suppose \eqref{eq:nonirre_comp_of_greatest_depth} is non-empty, and 
let $k$ be the maximum of \eqref{eq:nonirre_comp_of_greatest_depth}.
Since $F_k$ is reducible, there exists a $3$-ball $B$ in 
$\sphere\setminus \infty$
with $\partial B\subset  F_k$, 
$\partial B\cap \partial F_k=\emptyset$, 
and $\widehat{B\cap F_k}$ a prime $3$-manifold,
where $\widehat{B\cap F_k}$ is obtained by capping off 
the spherical component $\partial B$ of $\partial(B\cap F_k)$ 
with a $3$-ball.
 
Ignoring components of $\surface$ that are in $B$,
we obtain a new pair $\tilde{\SSS}$ with less solid parts.
By the induction hypothesis $\tilde{\SSS}$ 
admits a non-splittable graft decomposition: 
\[\tilde{\SSS}\simeq \SSS_1\dashleftarrow\cdots \dashleftarrow \SSS_q.\]
On the other hand, considering only
components of $\surface$ that are in $B$, we get
another pair $\overline{\SSS}=(\sphere, \overline{\surface})$, 
which is non-splittable.
Since $\SSS$ can be obtained by grafting $\overline{\SSS}$ onto
$\tilde{\SSS}$ at the solid part $F_k\cup B$ of $\tilde{\SSS}$, 
we have the non-splittable graft decomposition of $\SSS$:
\[\SSS
\simeq \SSS_1\dashleftarrow\cdots \dashleftarrow \SSS_q
\dashleftarrow \overline{\SSS}.\]

\textbf{Uniqueness:} 
Let $k$, $\tilde{\SSS}$, and $\overline{\SSS}$ 
be as above. Then we observe that     
$\overline{\SSS}$ can be identified with one of the  
$\SSS_i$, $i=1\dots m$ (resp.\ $\SSS_i'$, $i=1\dots p$) 
in \eqref{eq:non_split_graft_decomps}. 
Up to reindexing, we may assume 
\[\overline{\SSS}\simeq \SSS_m\simeq \SSS_p',\] 
and thus, \eqref{eq:non_split_graft_decomps}  
induces two non-splittable graft decompositions
of $\tilde{\SSS}$: 
\[
\SSS_1\dashleftarrow\dots
\dashleftarrow \SSS_{m-1}\quad and \quad 
\SSS_1'\dashleftarrow\dots
\dashleftarrow  \SSS_{p-1}'. 
\]
But $\tilde{\SSS}$ has fewer solid parts than
$\SSS$, so by the induction hypothesis, 
$m-1=p-1$, and after reindexing if necessary,
$\SSS_j\simeq \SSS_j'$ 
for $j=1\dots m-1$.  
\end{proof}

\noindent
\subsection{A detour: edge-labeled trees and trivial pairs.}
A based edge-labeled tree is a based tree with a non-negative integer
assigned to each edge. Given a pair $\SSS=\pair$, we label each edge of
the depth tree $\Lambda_\SSS$
by assigning to each edge the genus of the 
component of $\surface$ the edge represents. 
The resulting tree is called edge-labeled depth tree and 
denoted by $\Lambda_{\SSS}^\ast$.

\begin{definition}[\textbf{Isomorphism of edge-labeled trees}]
Two based edge-labeled trees $\Lambda_1^\ast$ and $\Lambda_2^\ast$ 
are isomorphic if there exists an isomorphism of based trees
between $\Lambda_1^\ast$ and $\Lambda_2^\ast$  
such that the labels of the corresponding edges are identical.
\end{definition}

\begin{definition}
A pair $\SSS$ has the type of a based edge-labeled tree $\Lambda^\ast$
if its edge-labeled depth tree $\Lambda_{\SSS}^\ast$ is isomorphic to
$\Lambda^\ast$.
\end{definition}

The simplest based non-degenerate tree is a based $1$-simplex;
a based $1$-simplex with a label $g$ is denoted by $\Lambda_0^g$. 
A pair $\pair$ of type $\Lambda_0^g$ is a connected surface 
$\Sigma=\surface$ of genus $g$
embedded in $\sphere$.

\begin{definition}[\textbf{Trivial pair}]
A pair $\SSS=(\sphere,\Sigma)$ of type $\Lambda_0^g$ is trivial 
if $\Sigma$ is trivially embedded in $\sphere$
(i.e.\ $\Sigma$ induces a Heegaard splitting of $\sphere$.).
A pair $\SSS=\pair$ of type $\Lambda^\ast$ is trivial if 
each $\SSS_i=(\sphere,\Sigma_i)$ in
the non-splittable decomposition of $\SSS$ 
\[\SSS\simeq \SSS_1\dashleftarrow\dots\dashleftarrow\SSS_n,\]
is a trivial pair of type $\Lambda_0^{g_i}$, 
where $g_i$ is the genus of $\Sigma_i$.
\end{definition}

\begin{proposition}[\textbf{Equivalence of trivial pairs}]
If $\Lambda^\ast$ is a based edge-labeled tree,
then any two trivial pairs $\SSS$, $\SSS'$ of type $\Lambda^\ast$
are equivalent.
\end{proposition}
\begin{proof}
We prove a slightly stronger statement: given an isomorphism 
$\mathcal{E}$ (resp.\ $\mathcal{E}'$) between
$\Lambda_\SSS^\ast$ and $\Lambda^\ast$ 
(resp.\ $\Lambda^\ast_{\SSS'}$ and $\Lambda^\ast$), 
the equivalence between $\SSS$ and $\SSS'$
can be chosen to respect $\mathcal{E}^{-1}\circ \mathcal{E}'$.

We prove it by induction on the number of nodes in $\Lambda^\ast$.
If $\Lambda^\ast$ has only one node, then 
$\surface$ is empty and the assertion
holds trivially. 

Suppose the statement is true for any
based edge-labeled tree with less than $m>1$ nodes and
$\Lambda^\ast$ has $m$ nodes. 
Then there exists a component $\Sigma_i=F_j\cap F_k$ of $\surface$ 
such that the solid part $F_k$ does not contain the base point and
any other component of $\surface$. 
Let $\Sigma_i'$, $F_j'$ and $F_k'$ be 
the component of $\surface'$ and solid parts of $\pairprime$ 
corresponding 
to $\widehat{jk}$, $j$, and $k$ 
via the equivalences $\mathcal{E}^{-1}\circ \mathcal{E}'$. 
Suppose   
\[
\SSS\simeq \SSS_1\dashleftarrow \SSS_2\dashleftarrow\dots
 \dashleftarrow \SSS_m\quad and \quad
\SSS'\simeq \SSS_1'\dashleftarrow \SSS_2'\dashleftarrow\dots
\dashleftarrow  \SSS_m'
\]
are the non-splittable graft decompositions of $\SSS$ and $\SSS'$; 
it may be assumed that $\SSS_m=(\sphere,\Sigma_i)$ 
and $\SSS_m=(\sphere,\Sigma_i')$. In other words, 
$\SSS$ and $\SSS'$ can be obtained 
by grafting $\SSS_m$ and $\SSS_{m}'$ onto 
\[
\tilde{\SSS}:=\SSS_1\dashleftarrow \SSS_2\dashleftarrow\dots
 \dashleftarrow \SSS_{m-1}\quad and \quad
\tilde{\SSS}':=\SSS_1'\dashleftarrow \SSS_2'\dashleftarrow\dots
\dashleftarrow  \SSS_{m-1}'
\]
at the solid parts of $\tilde{\SSS}$ and $\tilde{\SSS}'$ 
containing $F_j\cup F_k$ and $F_j'\cup F_k'$, respectively. 
Since $\SSS_m$ and $\SSS_m'$ are trivial pairs of type $\Lambda_0^g$, 
$\SSS_m$ and $\SSS_m'$ are equivalent by  
\cite{Wal:68i}. On the other hand,
$\tilde{\SSS}$ and $\tilde{\SSS}'$ is of type $\tilde{\Lambda}^\ast$,
which has one node less than $\Lambda^\ast$, 
and the isomorphism $\mathcal{E}$
(resp.\ $\mathcal{E}'$) induces 
an isomorphism between $\Lambda_{\tilde{\SSS}}^\ast$
and $\tilde{\Lambda}^\ast$ (resp.\ $\Lambda_{\tilde{\SSS}'}^\ast$ 
and $\tilde{\Lambda}^\ast$). By the induction hypothesis, 
there is an equivalence between
$\tilde{\SSS}$ and $\tilde{\SSS}'$ sending the solid part containing 
$F_j\cup F_k$ to the solid part containing $F_j'\cup F_k'$.  
Gluing this equivalence and the equivalence between 
$\SSS_m$ and $\SSS_m'$ together, we get an equivalence between
$\SSS$ and $\SSS'$. 
\end{proof}

%%%%%%%%%%%%%%%%%%%%%%%%%%%%%%%%%%%%%%%%%%%%%
%%Algebraic counterpart
%%%%%%%%%%%%%%%%%%%%%%%%%%%%%%%%%%%%%%%%%%%%%
\subsection{Algebraic graft decomposition}
%Recall that barycentric diagrams in a category $\mathsf{C}$ 
%are introduced in Definition \ref{def:bary_diag_in_C};  
Here, we consider a special case of barycentric diagrams
with $\mathsf{C}=\mathsf{Grp}_f$ (Definition \ref{def:bary_diag_in_C}).
Let $\Gamma$ be a based graph.
 
\begin{definition}[\textbf{Barycentric diagrams in $\mathsf{Grp}_f$ with pairing}]\label{def:bary_diagram_of_groups_w_pairing} 
A $\sd\Gamma$-diagram in $\mathsf{Grp}_f$ with pairing is 
a $\operatorname{sd}\Gamma$-diagram $\mathcal{G}$ in $\mathsf{Grp}_f$ 
together with a non-degenerate pairing
\begin{equation}\label{eq:pairing}
I: V_{\alpha\beta}\times V_{\alpha\beta}\rightarrow \mathbb{Z},
\end{equation} 
for every barycenter $\widehat{\alpha\beta}$ in $\operatorname{sd}\Gamma$,
where $V_{\alpha\beta}$ is the free abelian group given by
the free part of the abelianization of $\mathcal{G}(\widehat{\alpha\beta})$.
\end{definition}

\begin{definition}[\textbf{equivalence} $\simeq$ ]
Given a $\sd\Gamma$-diagram $\mathcal{G}$ with pairing and 
a $\sd\Gamma'$-diagram $\mathcal{G}'$ with pairing, 
we say that $\mathcal{G},\mathcal{G}'$ are equivalent if there exists a based equivalence 
$\mathcal{E}:\sd\Gamma\rightarrow \sd\Gamma'$ and a natural isomorphism
$\Phi:\mathcal{G}\Rightarrow \mathcal{G}'\circ \mathcal{E}$
such that $\Phi(\widehat{\alpha\beta})$ preserves 
pairings \eqref{eq:pairing}, for every barycenter $\widehat{\alpha\beta}$.
\end{definition}

We refer to a $\sd \Gamma$-diagram in $\mathsf{Grp}_f$ with pairing 
as a barycentric diagram in $\mathsf{Grp}_f$ with pairing 
when the type $\sd\Gamma$ is irrelevant in the discussion.
The set of equivalence classes of all barycentric diagrams in $\mathsf{Grp}_f$
with pairing is denoted by $\mathsf{Grp}_{f,p}^{\mathtt{BD}}$.

\begin{definition}[\textbf{Join of two finite graphs}] 
Let $\Gamma$ and $\Gamma'$ be two based finite graphs   
with base nodes $\ast$ and $\ast'$, respectively, 
and $i$ be a selected node in $\Gamma $. 
The join $\Gamma \vee_{i}\Gamma'$  
is a based graph
obtained by identifying $\ast'\in \Gamma'$ with $i\in \Gamma$
with $\ast$ the base node.  
\end{definition}

The barycentric subdivision of
$\Gamma \vee_{i}\Gamma'$   
can be identified with a pushout of 
\[\sd\Gamma \xleftarrow{i} \mathbf{1} \xrightarrow{\ast'} \sd\Gamma',\]
where $\mathbf{1}=\{1\}$ is the trivial category, and $i$ and $\ast'$
are functors sending $1$ to $i\in \Gamma$ and $1$ 
to $\ast'\in \Gamma'$, respectively.

\begin{definition}[\textbf{Grafting a barycentric diagram to another}]
Suppose $\mathcal{G}:\operatorname{sd}\Gamma \rightarrow \mathsf{Grp}_f$
and $\mathcal{G}':\operatorname{sd}\Gamma'\rightarrow \mathsf{Grp}_f$ 
are two barycentric diagrams in $\mathsf{Grp}_f$ with pairing. 
Then the barycentric diagram in $\mathsf{Grp}_f$ with pairing 
obtained by
grafting $\mathcal{G}'$ onto $\mathcal{G}$ at $\mathcal{G}(i)$ is 
the functor 
\[\mathcal{G}\overset{i}{\dashleftarrow}\mathcal{G}': 
\operatorname{sd}(\Gamma \vee_{i}\Gamma')\rightarrow  \mathsf{Grp}_f\] 
given by the assignment:
\begin{align*}
v & \mapsto \mathcal{G}(v) & v \in \sd \Gamma \setminus \{i\}\\
w &\mapsto \mathcal{G}'(w) & w \in \sd \Gamma'\setminus\{\ast'\}\\
u &\mapsto \mathcal{G}(i)\ast \mathcal{G}'(\ast')& u=[\ast']=[i],
\end{align*} 
where by $A\ast B$ we understand the free product of two
groups $A$ and $B$. The pairing of $\mathcal{G}\overset{i}{\dashleftarrow}\mathcal{G}'$ is the one inherited from $\mathcal{G},\mathcal{G}'$.  
\end{definition}

\begin{definition}[\textbf{Algebraic graft decomposition}]
Let  
$\mathcal{G}:\sd\Gamma \rightarrow\mathsf{Grp}_f$
be a barycentric diagram in $\mathsf{Grp}_f$ with pairing. 
Suppose 
$\Gamma=\Gamma_1\vee_{i_1}\Gamma_2\vee_{i_2}\dots\vee_{i_{n-1}}\Gamma_n$ 
and 
\begin{equation}\label{eq:alg_graft_decomp}
\mathcal{G}\simeq \mathcal{G}_1\overset{i_1}{\dashleftarrow} \mathcal{G}_2
\overset{i_2}{\dashleftarrow}\dots \overset{i_{n-1}}{\dashleftarrow}\mathcal{G}_n,
\end{equation}
where $i_k\in\Gamma_1\vee_{i_1}\Gamma_2\vee_{i_2}\dots\vee_{i_{k-1}}\Gamma_k$ and $\mathcal{G}_k:\sd\Gamma_k\rightarrow \mathsf{Grp}_f$, 
is a barycentric diagram in $\mathsf{Grp}_f$ with pairing, $k=1,\cdots, n$. 
Then \eqref{eq:alg_graft_decomp} is called a graft decomposition of $\mathcal{G}$.
\end{definition}

Given a pair $\SSS=\pair$, 
we orient a component $\Sigma_i$ of $\surface$
such that 
its normal vectors point toward the
component of $\sphere\setminus \Sigma_i$ containing $\infty$.
The orientation induces an intersection form on $H_1(\Sigma_i)$,
the abelianization of $\mathcal{FT}^u(\SSS)(\widehat{jk})$,
where $\Sigma_i=F_j\cap F_k$.
  
\begin{definition}[\textbf{Fundamental tree}]
The fundamental tree $\mathcal{FT}(\SSS)$ of a pair $\SSS=\pair$ 
is a barycentric diagram in $\mathsf{Grp}_f$ with pairing
given by the functor $\mathcal{FT}^u(\SSS)$ 
together with the intersection form on
$H_1(\Sigma_i)$ for each $i$. 
%$\mathcal{FT}(\SSS)$  denotes
%the fundamental tree of $\SSS$.
\end{definition}

The fundamental tree induces a mapping 
\[\mathcal{FT}:\mathsf{Sur}\rightarrow \mathsf{Grp}_{f,p}^{\mathtt{BD}},\]
and Theorem \ref{intro:main_thm} asserts that $\mathcal{FT}(\SSS)$
is a complete invariant of $\SSS$, that is, 
%\begin{theorem}\label{Thm:a_complete_invariant}
the mapping $\mathcal{FT}$ is injective.
%\end{theorem}
%The proof of Theorem \ref{Thm:a_complete_invariant} is given in 
%Section \ref{sec:proof}.  

\section{Generalized Kneser's conjecture}\label{sec:threemfd}
%Main references are \cite{Fox:48}, \cite{Wal:67}, \cite{Wal:68ii}, \cite{Hei:71},
%\cite{Suz:75} and \cite{Hem:04}.
\begin{proof}[Proof of Theorem \ref{intro:teo:g_kneser_conjecture}]
Note first if one of $A_1,A_2$, say $A_1$, is trivial,
then $\surface_1$ must be empty. 
This can be seen from the exact sequence
\begin{equation}\label{eq:non-trivial:exact_seq_for_Sigma_M}
H_2(M,\partial M;\mathbb{Z}_2)\xrightarrow{\partial} H_1(\partial M;\mathbb{Z}_2)\rightarrow H_1(M;\mathbb{Z}_2).
\end{equation}
By \eqref{eq:non-trivial:exact_seq_for_Sigma_M}, 
any two loops in 
%$H_1(\partial M,\mathbb{Z}_2)$ 
%coming from $H_2(M,\partial M;\mathbb{Z}_2)$
the image of $\partial$ have null intersection number.
So, if $\surface_1\neq\emptyset$, then   
the induced homomorphism
\[H_1(\surface_{1};\mathbb{Z}_2)\rightarrow H_1(M;\mathbb{Z}_2)\]
is non-trivial, and hence the homomorphism
\[\pi_1(\Sigma_{1j})\xrightarrow{\delta_{1j}} \pi_1(M)\]
is non-trivial, for every $j$, contradicting 
the triviality of $A_1$. 
%since $\phi$ is an isomorphism. 
If $\surface_1=\emptyset$, the theorem follows trivially.

In the case where both $A_1$ and $A_2$ are non-trivial,
we divide the proof into three parts, and
assume that $M$ contains no spherical components. 
The connected sum decomposition of $M$
is constructed in Step $1$, $\partial M_i=\surface_i$, $i=1,2$,
is verified in Step $2$, 
and the commutative diagram \eqref{diag:diag_in_g_Kneser_lemma} 
is proved in Step $3$.

\textbf{Step $1$: connected sum decomposition of $M$}. 
Following the methods in \cite{Sta:65}, 
\cite[Lemma]{Hei:71}, and \cite[Chap.\ $7$]{Hem:04}, we consider two 
aspherical $\op{CW}$-complexes $K_1$ and $K_2$ with 
$\pi_{1}(K_i)\simeq A_i$.
Connecting $K_1$ and $K_2$ to an interval $I=[0,1]$ by
gluing $0,1\in I$ to base points of $K_1, K_2$, respectively,
we obtain a new $\op{CW}$-complex $K$. 
Let $\ast:=\frac{1}{2}\in I$ be the base point of $K$. 
Then there is an obvious isomorphism
$\pi_1(K)\simeq A_1\ast A_2$,
and hence $\phi$ can be viewed as an isomorphism
from $\pi_1(M)$ to $\pi_1(K)$. 
Since $K$ is aspherical, the isomorphism
can be realized by a map $h:M\rightarrow K$.

By \cite[Lemma $1.1$]{Wal:67}, we may assume $i)$ $h$ 
is transverse to $\ast$ (i.e.\ $h^{-1}(I')$ 
has a product structure $h^{-1}(\ast)\times I'$ 
on which $h$ restricts to the projection onto $I'$, where $I'=[\frac{1}{4},\frac{3}{4}]\subset I$),
and $ii)$ $h^{-1}(\ast)$ consists of incompressible surfaces. 
Since $h_\ast=\phi$ is an isomorphism, 
a component in $h^{-1}(\ast)$
is either a disk or a $2$-sphere. Let $(n_d,n_s)$
denote the numbers of disks and spheres in $h^{-1}(\ast)$. 
We define a linear order $\preceq$ on the set of 
pairs of non-negative integers 
by declaring $(a,b)\preceq (c,d)$ if either $a<c$ or $b<d$ when $a=c$. 
We assume $h$ is so chosen that it satisfies conditions $i)$ and $ii)$ 
and minimizes $(n_d,n_s)$. 
\textit{We shall see $(n_d,n_s)$ of $h$ is $(0,1)$, 
and the $2$-sphere $h^{-1}(\ast)$ induces
the required connected sum.}

\textbf{Disks:}
Observe first that the set of boundary circles 
$\{\partial D_k\}_{k=1}^{n_d}$ of disks in $h^{-1}(\ast)$    
separates $\partial M$ 
into several components, and crossing a boundary $\partial D_k$  
means going from one component to the other.   
Thinking of each $\partial D_k$ as 
an edge and the closure of each component of the complement
\[\partial M\setminus \bigcup_{k=1}^{n_d} \partial D_k\] 
as a node, we obtain a graph $\mathbf{G}$ (Fig.\ \ref{fig:separating_disks_graph_G}). 
\begin{figure}[ht]
\begin{subfigure}{0.5\textwidth}
\centering
\def\svgwidth{0.95\columnwidth}
%% Accompanies image file 'separating_disks.pdf' (pdf, eps, ps)
\begingroup%
  \makeatletter%
  \providecommand\color[2][]{%
    \errmessage{(Inkscape) Color is used for the text in Inkscape, but the package 'color.sty' is not loaded}%
    \renewcommand\color[2][]{}%
  }%
  \providecommand\transparent[1]{%
    \errmessage{(Inkscape) Transparency is used (non-zero) for the text in Inkscape, but the package 'transparent.sty' is not loaded}%
    \renewcommand\transparent[1]{}%
  }%
  \providecommand\rotatebox[2]{#2}%
  \newcommand*\fsize{\dimexpr\f@size pt\relax}%
  \newcommand*\lineheight[1]{\fontsize{\fsize}{#1\fsize}\selectfont}%
  \ifx\svgwidth\undefined%
    \setlength{\unitlength}{1559.05511811bp}%
    \ifx\svgscale\undefined%
      \relax%
    \else%
      \setlength{\unitlength}{\unitlength * \real{\svgscale}}%
    \fi%
  \else%
    \setlength{\unitlength}{\svgwidth}%
  \fi%
  \global\let\svgwidth\undefined%
  \global\let\svgscale\undefined%
  \makeatother%
  \begin{picture}(1,0.45454545)%
    \lineheight{1}%
    \setlength\tabcolsep{0pt}%
    \put(0,0){\includegraphics[width=\unitlength,page=1]{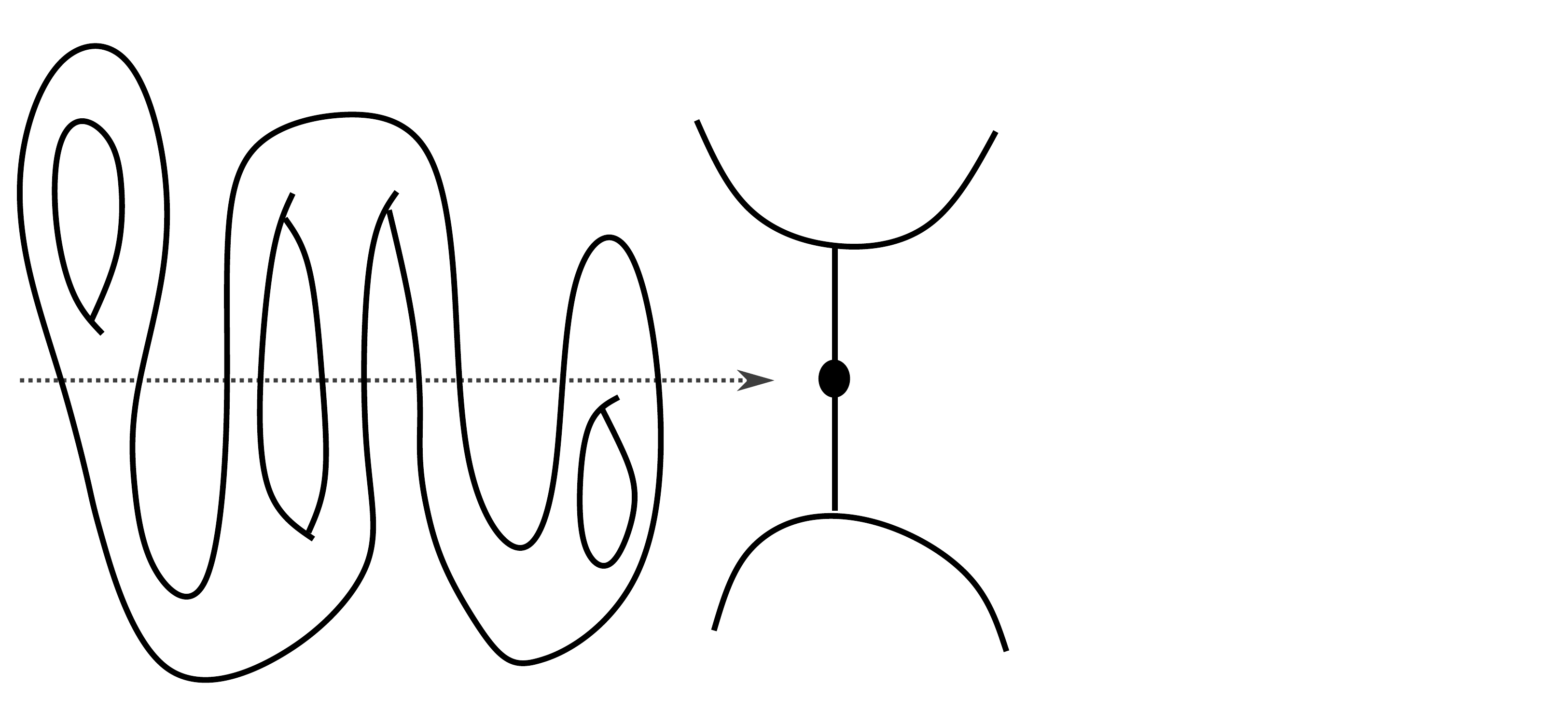}}%
    \put(0.50070597,0.39314247){\color[rgb]{0,0,0}\makebox(0,0)[lt]{\lineheight{1.25}\smash{\begin{tabular}[t]{l}$K_1$\end{tabular}}}}%
    \put(0.49474862,0.00497804){\color[rgb]{0,0,0}\makebox(0,0)[lt]{\lineheight{1.25}\smash{\begin{tabular}[t]{l}$K_2$\end{tabular}}}}%
    \put(0.55809156,0.20031775){\color[rgb]{0,0,0}\makebox(0,0)[lt]{\lineheight{1.25}\smash{\begin{tabular}[t]{l}$\ast$\end{tabular}}}}%
    \put(0,0){\includegraphics[width=\unitlength,page=2]{separating_disks.pdf}}%
    \put(0.66552085,0.206137){\color[rgb]{0,0,0}\makebox(0,0)[lt]{\lineheight{1.25}\smash{\begin{tabular}[t]{l}$\Rightarrow$\end{tabular}}}}%
    \put(0,0){\includegraphics[width=\unitlength,page=3]{separating_disks.pdf}}%
    \put(0.10236823,0.41864018){\color[rgb]{0,0,0}\makebox(0,0)[lt]{\lineheight{1.25}\smash{\begin{tabular}[t]{l}$\surface$\end{tabular}}}}%
    \put(0.84797563,0.078304){\color[rgb]{0,0,0}\makebox(0,0)[lt]{\lineheight{1.25}\smash{\begin{tabular}[t]{l}$\mathbf{G}$\end{tabular}}}}%
  \end{picture}%
\endgroup%
  
\caption{The graph $\mathbf{G}$.}
\label{fig:separating_disks_graph_G}
\end{subfigure}
\begin{subfigure}{0.45\textwidth}
\centering
\def\svgwidth{0.95\columnwidth}
%% Accompanies image file 'kneser_no_loops.pdf' (pdf, eps, ps)
\begingroup%
  \makeatletter%
  \providecommand\color[2][]{%
    \errmessage{(Inkscape) Color is used for the text in Inkscape, but the package 'color.sty' is not loaded}%
    \renewcommand\color[2][]{}%
  }%
  \providecommand\transparent[1]{%
    \errmessage{(Inkscape) Transparency is used (non-zero) for the text in Inkscape, but the package 'transparent.sty' is not loaded}%
    \renewcommand\transparent[1]{}%
  }%
  \providecommand\rotatebox[2]{#2}%
  \newcommand*\fsize{\dimexpr\f@size pt\relax}%
  \newcommand*\lineheight[1]{\fontsize{\fsize}{#1\fsize}\selectfont}%
  \ifx\svgwidth\undefined%
    \setlength{\unitlength}{1133.85826772bp}%
    \ifx\svgscale\undefined%
      \relax%
    \else%
      \setlength{\unitlength}{\unitlength * \real{\svgscale}}%
    \fi%
  \else%
    \setlength{\unitlength}{\svgwidth}%
  \fi%
  \global\let\svgwidth\undefined%
  \global\let\svgscale\undefined%
  \makeatother%
  \begin{picture}(1,0.625)%
    \lineheight{1}%
    \setlength\tabcolsep{0pt}%
    \put(0,0){\includegraphics[width=\unitlength,page=1]{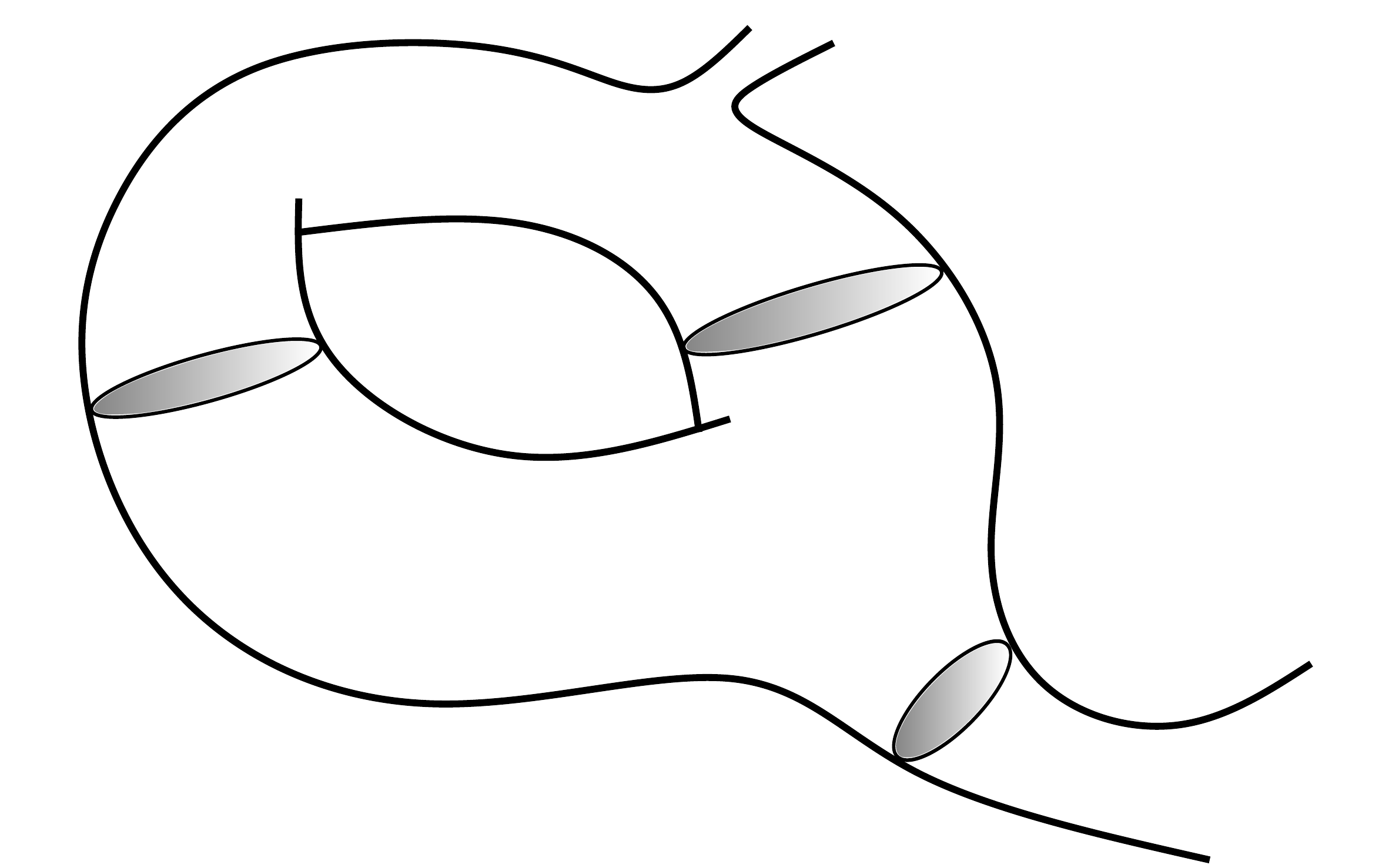}}%
    \put(0.85768738,0.05797333){\color[rgb]{0,0,0}\makebox(0,0)[lt]{\lineheight{1.25}\smash{\begin{tabular}[t]{l}{\tiny $\ast_\Sigma$}\end{tabular}}}}%
    \put(0,0){\includegraphics[width=\unitlength,page=2]{kneser_no_loops.pdf}}%
    \put(0.45266834,0.22156925){\color[rgb]{0,0,0}\makebox(0,0)[lt]{\lineheight{1.25}\smash{\begin{tabular}[t]{l}{\tiny $l_1$}\end{tabular}}}}%
    \put(0.22024811,0.28459587){\color[rgb]{0,0,0}\makebox(0,0)[lt]{\lineheight{1.25}\smash{\begin{tabular}[t]{l}{\tiny $l_2$}\end{tabular}}}}%
    \put(0.15433323,0.45772167){\color[rgb]{0,0,0}\makebox(0,0)[lt]{\lineheight{1.25}\smash{\begin{tabular}[t]{l}{\tiny $l_3$}\end{tabular}}}}%
    \put(0.47676018,0.49901773){\color[rgb]{0,0,0}\makebox(0,0)[lt]{\lineheight{1.25}\smash{\begin{tabular}[t]{l}{\tiny $l_4$}\end{tabular}}}}%
  \end{picture}%
\endgroup%
  
\caption{$l$ and the connecting arc $\gamma$.}
\label{fig:no_loop_l_gamma}
\end{subfigure}
\caption{}
\end{figure}
 
\textbf{Claim $1$: No loop in $\mathbf{G}$.}
We argue by contradiction. 
Suppose there is a loop in a component $G$ of $\mathbf{G}$. 
Without loss of generality, it may be assumed
that $\Sigma:=\Sigma_{11}\subset\surface_1$ 
is the component corresponding to $G$.
 
%that the boundary of $D:=D_1$ represents an edge in the loop, and 
%that $\partial D$ is in the component $\Sigma:=\Sigma_{11}$ of $\surface_1$.

The loop induces an embedded 
loop $l$ transverse to $\bigcup_{k=1}^{n_d} \partial D_k$ in $\Sigma$  
such that 
%$l\cap D$ is a point and 
$l\cap \big(\bigcup_{k=1}^{n_d} D_k\big)\neq \emptyset$ 
and $l\cap D_k$ contains no more than one point for each $k$.
In particular, $l$ is dual to $\partial D_k$ in $\Sigma$, for some 
$k$, and hence, is
essential in $M$ by \eqref{eq:non-trivial:exact_seq_for_Sigma_M}.   

Now, $l\cap h^{-1}(\ast)$ divides $l$
into $2n$ arcs $l_1$,\dots,$l_{2n}$ with end points of 
each $l_i$ lying  
in different disks (Fig.\ \ref{fig:no_loop_l_gamma}). 
Up to reindexing,
%$l_1,\dots,l_{2n}$
we may assume   
that $h\circ l_i$ is a loop in $K_1$ when $i$ is odd,
and is a loop in $K_2$ when $i$ is even. 
Connect the base point $\ast_\Sigma$ 
of $\Sigma$ to $l$ via an arc $\gamma$.
Then the composition \eqref{eq:homo_bdry_to_M} 
sends $\gamma l\gamma^{-1}$ to an element 
$y x_1  x_2\dots x_{2n-1}  x_{2n}  y^{-1}\in A_1$, where 
$y\in A_1 \ast A_2$ is induced by $\gamma$, and
$x_i=[h\circ l_i]$ in $A_j$, $i\equiv j$ (mod $2$). In particular, 
some $x_i$, say $x_1$, must be trivial.
 
It may be assumed that $\partial l_1={p,q}$ are in 
the boundary of disks $D_1,D_2\in h^{-1}(\ast)$, respectively. 
%Then we homotopy a neighborhood $l_1$ 
Construct a map $h':M\rightarrow K$ that satisfies 
\begin{equation}\label{eq:hprime_condition} 
\begin{cases}
h'(x)=h(x) &  x\in M\setminus \mathring{\mathfrak{N}},\\
h'(D_3)=\ast:=\{\frac{1}{2}\}\in I=[0,1],\\
h'(D_1\cup D_2 \cup l_1)=\{\frac{2}{3}\}\subset I,\\
h'(I_i)=[\frac{1}{4},\frac{3}{4}]\subset I,
\end{cases}
\end{equation}
where $\mathring{\mathfrak{N}}$ is the union of open regular neighborhoods 
$\mathring{\mathfrak{N}}(D_i)$, $\mathring{\mathfrak{N}}(l_1)$ of $D_i$, $i=1,2$,
$l_1$, respectively, such that 
$h(\mathring{\mathfrak{N}}(D_i))=(\frac{1}{4},\frac{3}{4})$, and 
$D_3$ is the boundary component of a regular neighborhood $\mathfrak{N}'\subset
\mathring{\mathfrak{N}}$ 
of $D_1\cup D_2\cup l_1$ 
%in $\mathring{\mathfrak{N}}$
not parallel to $D_1$ or $D_2$, and 
$I_i\subset 
\partial  \mathfrak{N}(D_i) \cap \partial  \mathfrak{N}$  
is an arc connecting the closures of
components of $\partial \mathfrak{N}(D_i)-\partial M$ 
with $h(I_i)=[\frac{1}{4},\frac{3}{4}]$, $i=1,2$.
Because $h\circ l_1$ is null homotopic, 
by obstruction theory \eqref{eq:hprime_condition}
can be extended to a global map $h'$ from $M$ to $K$ 
such that
$h^{'-1}(\ast)$ contains the same $2$-spheres
but one disk less since disks $D_1,D_2$ are now merged into $D_3$
via a $\bar{\nabla}$-move along $l_1$ \cite[Section $3$]{Suz:75}, \cite{Joh:91} (Fig.\ \ref{fig:homotopy}).
It is not difficult to see that $h'$ is homotopic 
to $h$---because $h\circ l_1$ is null-homotopic. 
This contradicts
%in $\Sigma$, using a homotopy similar to that in \cite[p. 507]{Wal:67} 
%  such that $D_i,D_j\subset h^{-1}(\ast)$ 
%is replaced by a disk obtained by performing
%a $\bar{\nabla}$-move along $l_1$ (\cite[Section $3$]{Suz:75}, \cite{Joh:91}). 
the minimality of $\# h^{-1}(\ast)$, and thus, $\mathbf{G}$ is 
a union of trees. 
\begin{figure}[ht]
\begin{subfigure}{0.48\textwidth}
\centering
\def\svgwidth{0.95\columnwidth}
%% Accompanies image file 'homotopy_delta.pdf' (pdf, eps, ps)
\begingroup%
  \makeatletter%
  \providecommand\color[2][]{%
    \errmessage{(Inkscape) Color is used for the text in Inkscape, but the package 'color.sty' is not loaded}%
    \renewcommand\color[2][]{}%
  }%
  \providecommand\transparent[1]{%
    \errmessage{(Inkscape) Transparency is used (non-zero) for the text in Inkscape, but the package 'transparent.sty' is not loaded}%
    \renewcommand\transparent[1]{}%
  }%
  \providecommand\rotatebox[2]{#2}%
  \newcommand*\fsize{\dimexpr\f@size pt\relax}%
  \newcommand*\lineheight[1]{\fontsize{\fsize}{#1\fsize}\selectfont}%
  \ifx\svgwidth\undefined%
    \setlength{\unitlength}{1417.32283465bp}%
    \ifx\svgscale\undefined%
      \relax%
    \else%
      \setlength{\unitlength}{\unitlength * \real{\svgscale}}%
    \fi%
  \else%
    \setlength{\unitlength}{\svgwidth}%
  \fi%
  \global\let\svgwidth\undefined%
  \global\let\svgscale\undefined%
  \makeatother%
  \begin{picture}(1,0.5)%
    \lineheight{1}%
    \setlength\tabcolsep{0pt}%
    \put(0,0){\includegraphics[width=\unitlength,page=1]{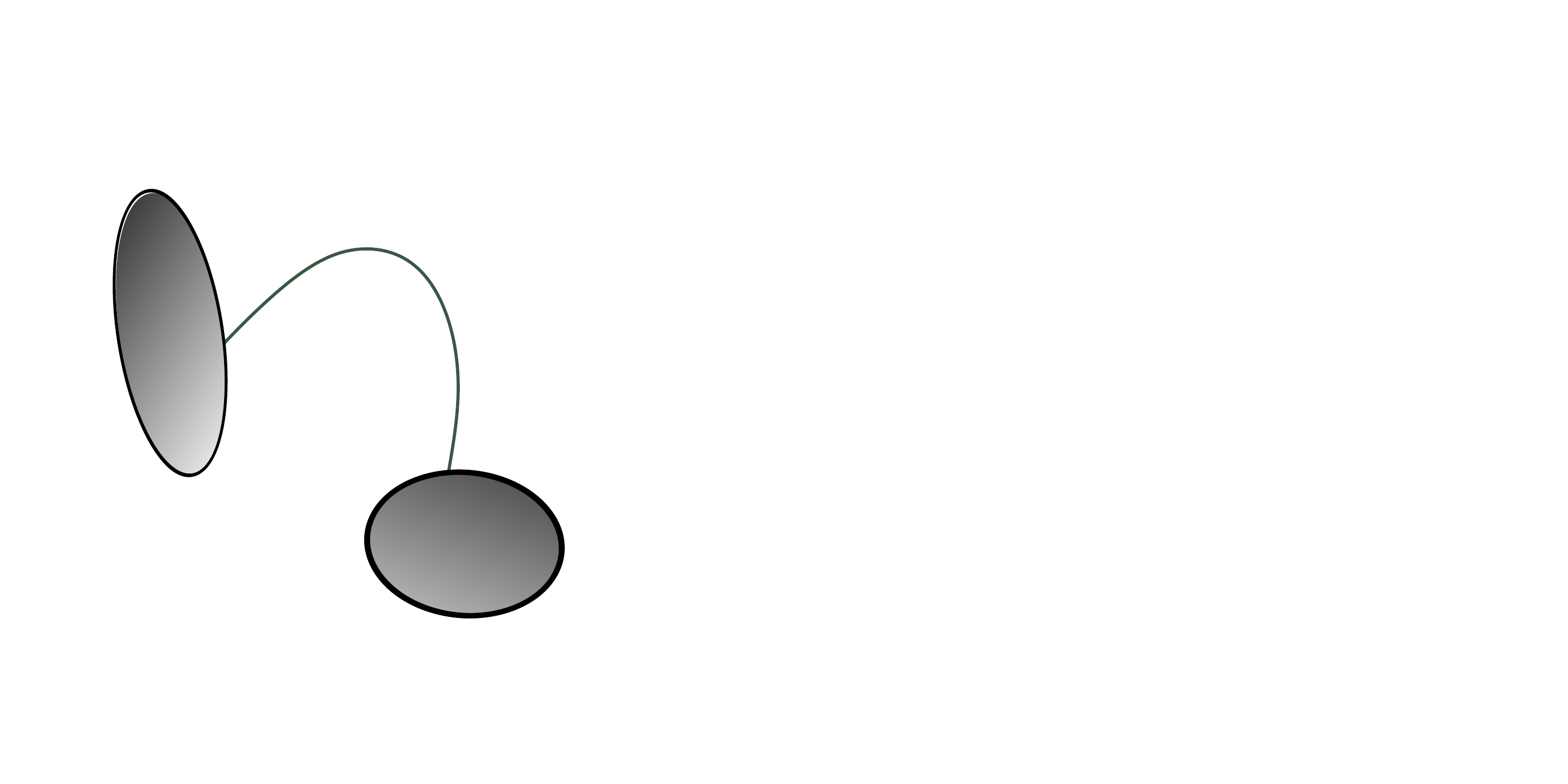}}%
    \put(0.25448619,0.34444542){\color[rgb]{0,0,0}\makebox(0,0)[lt]{\lineheight{1.25}\smash{\begin{tabular}[t]{l}{\tiny $l_1$}\end{tabular}}}}%
    \put(0,0){\includegraphics[width=\unitlength,page=2]{homotopy_delta.pdf}}%
    \put(0.10463359,0.15724021){\color[rgb]{0,0,0}\makebox(0,0)[lt]{\lineheight{1.25}\smash{\begin{tabular}[t]{l}{\tiny $D_1$}\end{tabular}}}}%
    \put(0.36313253,0.15810576){\color[rgb]{0,0,0}\makebox(0,0)[lt]{\lineheight{1.25}\smash{\begin{tabular}[t]{l}{\tiny $D_2$}\end{tabular}}}}%
    \put(0,0){\includegraphics[width=\unitlength,page=3]{homotopy_delta.pdf}}%
    \put(0.79687036,0.36222659){\color[rgb]{0,0,0}\makebox(0,0)[lt]{\lineheight{1.25}\smash{\begin{tabular}[t]{l}{\tiny $D_3$}\end{tabular}}}}%
    \put(0,0){\includegraphics[width=\unitlength,page=4]{homotopy_delta.pdf}}%
  \end{picture}%
\endgroup%
  
\caption{$\bar{\nabla}$-move along $l_1$.}
\label{fig:homotopy}
\end{subfigure}
\begin{subfigure}{0.48\textwidth}
\centering
\def\svgwidth{0.95\columnwidth}
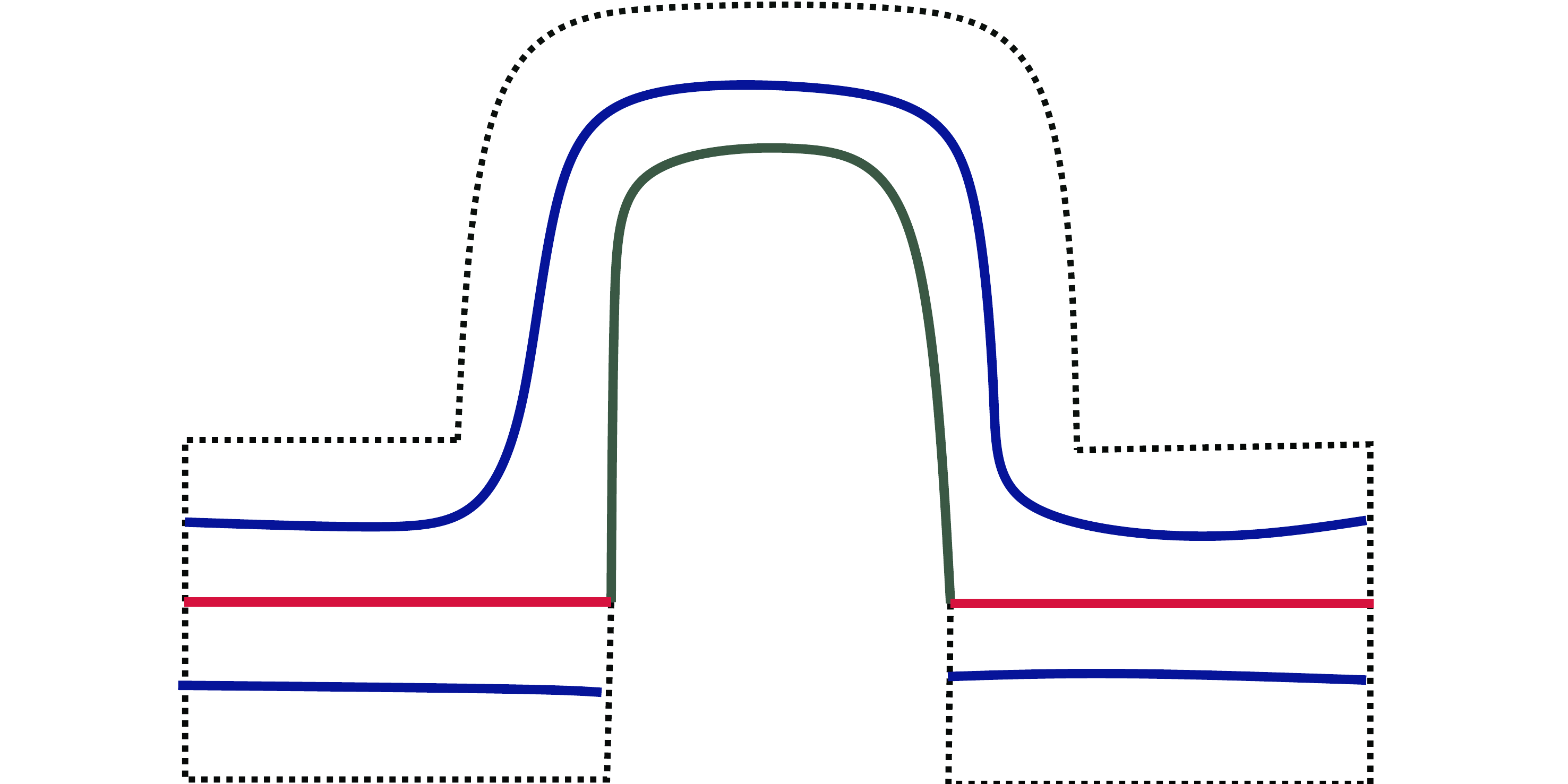  
\caption{The map $h'$ near $D_1\cup D_2\cup l_1$.}
\label{fig:hprime}
\end{subfigure}
\caption{Construction of $h'$.}
\end{figure} 
%%%%%%%%%%%%%%%%%%%Change the G above into bold letter

Recall that a $1$-valent node
in a graph is called a leaf. 
%A leaf of $\mathbf{G}$ corresponds to a component of 
%$\partial M \setminus \bigcup_{i=1}^{n_k} \partial D_i$, which 
%is adjacent to only another component. 
We define the genus of a leaf in $\mathbf{G}$
to be the genus of the corresponding component.

%%%%%%%%%%%%%%%%%%%%%%%%%%%%%%%%%%%%Claim 2 

\textbf{Claim $2$: At most one positive genus leaf  
in each component of $\mathbf{G}$.}
Suppose there are two leaves with genus greater than $0$ in a component 
$G$ of $\mathbf{G}$. 
Denote by $\Sigma$ the component of $\surface$
corresponding to $G$ and by 
$\Sigma', \Sigma''$ components of
$\Sigma\setminus \bigcup_{i=1}^{n_k} \partial D_i$ corresponding to
the two positive genus leaves. 
Without loss of generality, it may be assumed that
\begin{equation}\label{eq:homo_boundary_to_A_1}
\pi_1(\Sigma)\xrightarrow{\delta}\pi_1(M)\xrightarrow{\phi} 
A_1\ast A_2
\end{equation}  
factors through $A_1\rightarrow A_1\ast A_2$. 
%$\Sigma=\Sigma_{11}$ and abbreviate $\delta_{11}$ to $\delta$. 
%
Since $\Sigma'$ (resp.\ $\Sigma''$)
has positive genus, there exists an essential loop $l'$ 
(resp.\ $l''$) in $\Sigma'$ (resp.\ $\Sigma''$) 
which is also essential in $M$ by 
\eqref{eq:non-trivial:exact_seq_for_Sigma_M}.

Suppose $h(\Sigma')$ is in the component of 
$K\setminus \ast$ containing $K_2$ (Fig.\ \ref{fig:no_positive_leaf_K_2}). 
Then $h\circ l'$ is essential in $K_2$. 
On the other hand, the composition \eqref{eq:homo_boundary_to_A_1}
factors through $A_1\rightarrow A_1\ast A_2$ implies 
$\phi\circ\delta [\gamma l'\gamma^{-1}]=yxy^{-1}\in A_1$,
where $x\neq 1$ in $A_2$ is induced by $l'$, and
%=[h\circ l']\neq 1\in A_2$, 
$y\in A_1\ast A_2$ is induced by the arc $\gamma$ 
connecting $\ast_\Sigma$ to $l'$. 
This is possible, however, only when $x=1$. 
Similarly, one can show that 
$h(\Sigma'')$ cannot be in the 
component of $K\setminus \ast$ containing $K_2$.
%%add some figures
%%
\begin{figure}[ht]
\begin{subfigure}{0.48\textwidth}
\centering
\def\svgwidth{0.95\columnwidth}
%% Accompanies image file 'kneser_no_positive_leaf_in_K_2.pdf' (pdf, eps, ps)
\begingroup%
  \makeatletter%
  \providecommand\color[2][]{%
    \errmessage{(Inkscape) Color is used for the text in Inkscape, but the package 'color.sty' is not loaded}%
    \renewcommand\color[2][]{}%
  }%
  \providecommand\transparent[1]{%
    \errmessage{(Inkscape) Transparency is used (non-zero) for the text in Inkscape, but the package 'transparent.sty' is not loaded}%
    \renewcommand\transparent[1]{}%
  }%
  \providecommand\rotatebox[2]{#2}%
  \newcommand*\fsize{\dimexpr\f@size pt\relax}%
  \newcommand*\lineheight[1]{\fontsize{\fsize}{#1\fsize}\selectfont}%
  \ifx\svgwidth\undefined%
    \setlength{\unitlength}{992.12598425bp}%
    \ifx\svgscale\undefined%
      \relax%
    \else%
      \setlength{\unitlength}{\unitlength * \real{\svgscale}}%
    \fi%
  \else%
    \setlength{\unitlength}{\svgwidth}%
  \fi%
  \global\let\svgwidth\undefined%
  \global\let\svgscale\undefined%
  \makeatother%
  \begin{picture}(1,0.71428571)%
    \lineheight{1}%
    \setlength\tabcolsep{0pt}%
    \put(0,0){\includegraphics[width=\unitlength,page=1]{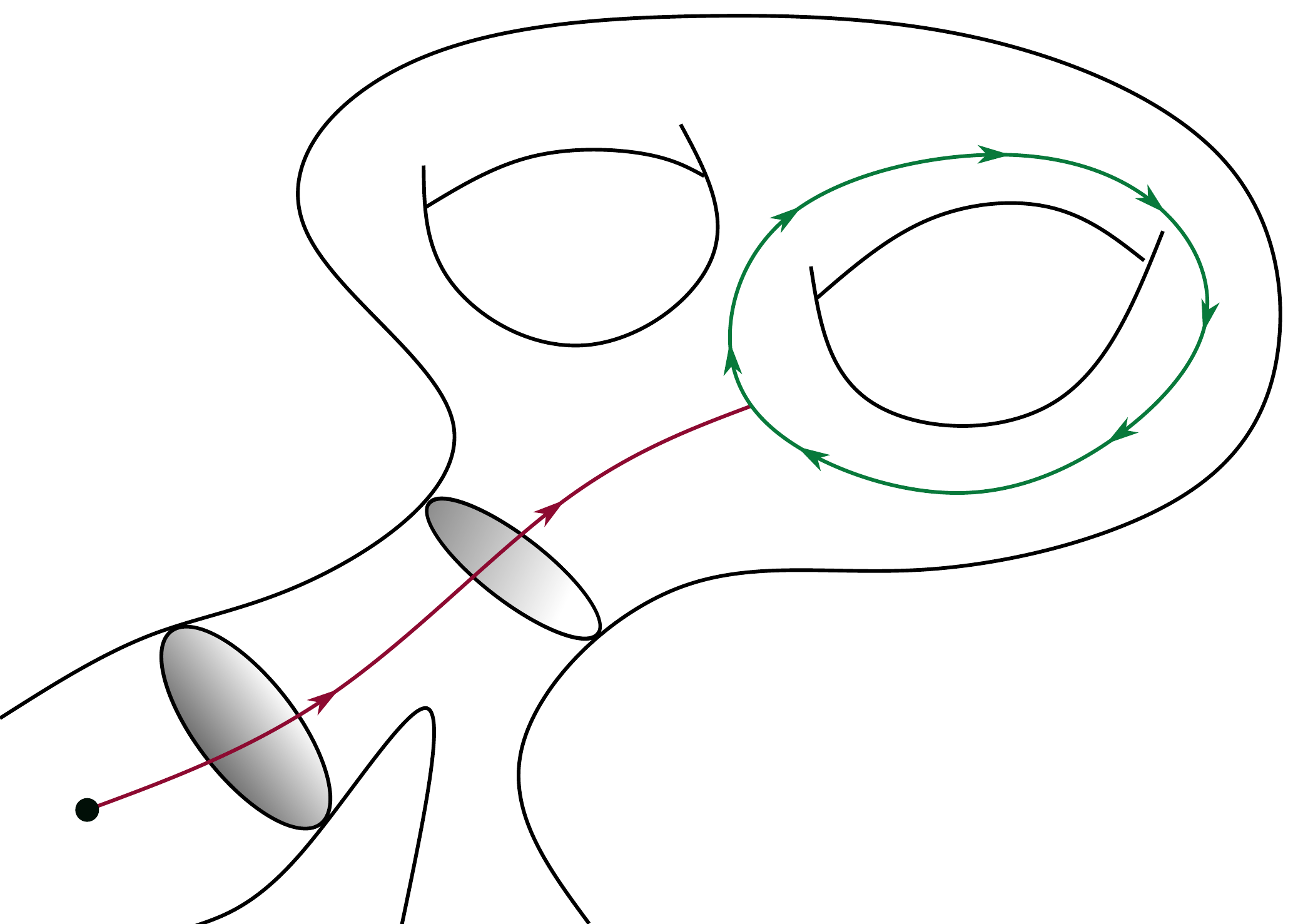}}%
    \put(0.06988564,0.05899449){\color[rgb]{0,0,0}\makebox(0,0)[lt]{\lineheight{1.25}\smash{\begin{tabular}[t]{l}{\tiny $\ast_\Sigma$}\end{tabular}}}}%
    \put(0.62534031,0.30767831){\color[rgb]{0,0,0}\makebox(0,0)[lt]{\lineheight{1.25}\smash{\begin{tabular}[t]{l}{\tiny $l'$}\end{tabular}}}}%
    \put(0.1742603,0.66709021){\color[rgb]{0,0,0}\makebox(0,0)[lt]{\lineheight{1.25}\smash{\begin{tabular}[t]{l}{\tiny $\Sigma'$}\end{tabular}}}}%
    \put(0.39480848,0.34852064){\color[rgb]{0,0,0}\makebox(0,0)[lt]{\lineheight{1.25}\smash{\begin{tabular}[t]{l}{\tiny $\gamma$}\end{tabular}}}}%
  \end{picture}%
\endgroup%
  
\caption{$l'$ and the connecting arc $\gamma$.}
\label{fig:no_positive_leaf_K_2}
\end{subfigure}
\begin{subfigure}{0.48\textwidth}
\centering
\def\svgwidth{0.95\columnwidth}
%% Accompanies image file 'kneser_no_positive_leaf_in_K_1.pdf' (pdf, eps, ps)
\begingroup%
  \makeatletter%
  \providecommand\color[2][]{%
    \errmessage{(Inkscape) Color is used for the text in Inkscape, but the package 'color.sty' is not loaded}%
    \renewcommand\color[2][]{}%
  }%
  \providecommand\transparent[1]{%
    \errmessage{(Inkscape) Transparency is used (non-zero) for the text in Inkscape, but the package 'transparent.sty' is not loaded}%
    \renewcommand\transparent[1]{}%
  }%
  \providecommand\rotatebox[2]{#2}%
  \newcommand*\fsize{\dimexpr\f@size pt\relax}%
  \newcommand*\lineheight[1]{\fontsize{\fsize}{#1\fsize}\selectfont}%
  \ifx\svgwidth\undefined%
    \setlength{\unitlength}{992.12598425bp}%
    \ifx\svgscale\undefined%
      \relax%
    \else%
      \setlength{\unitlength}{\unitlength * \real{\svgscale}}%
    \fi%
  \else%
    \setlength{\unitlength}{\svgwidth}%
  \fi%
  \global\let\svgwidth\undefined%
  \global\let\svgscale\undefined%
  \makeatother%
  \begin{picture}(1,0.71428571)%
    \lineheight{1}%
    \setlength\tabcolsep{0pt}%
    \put(0,0){\includegraphics[width=\unitlength,page=1]{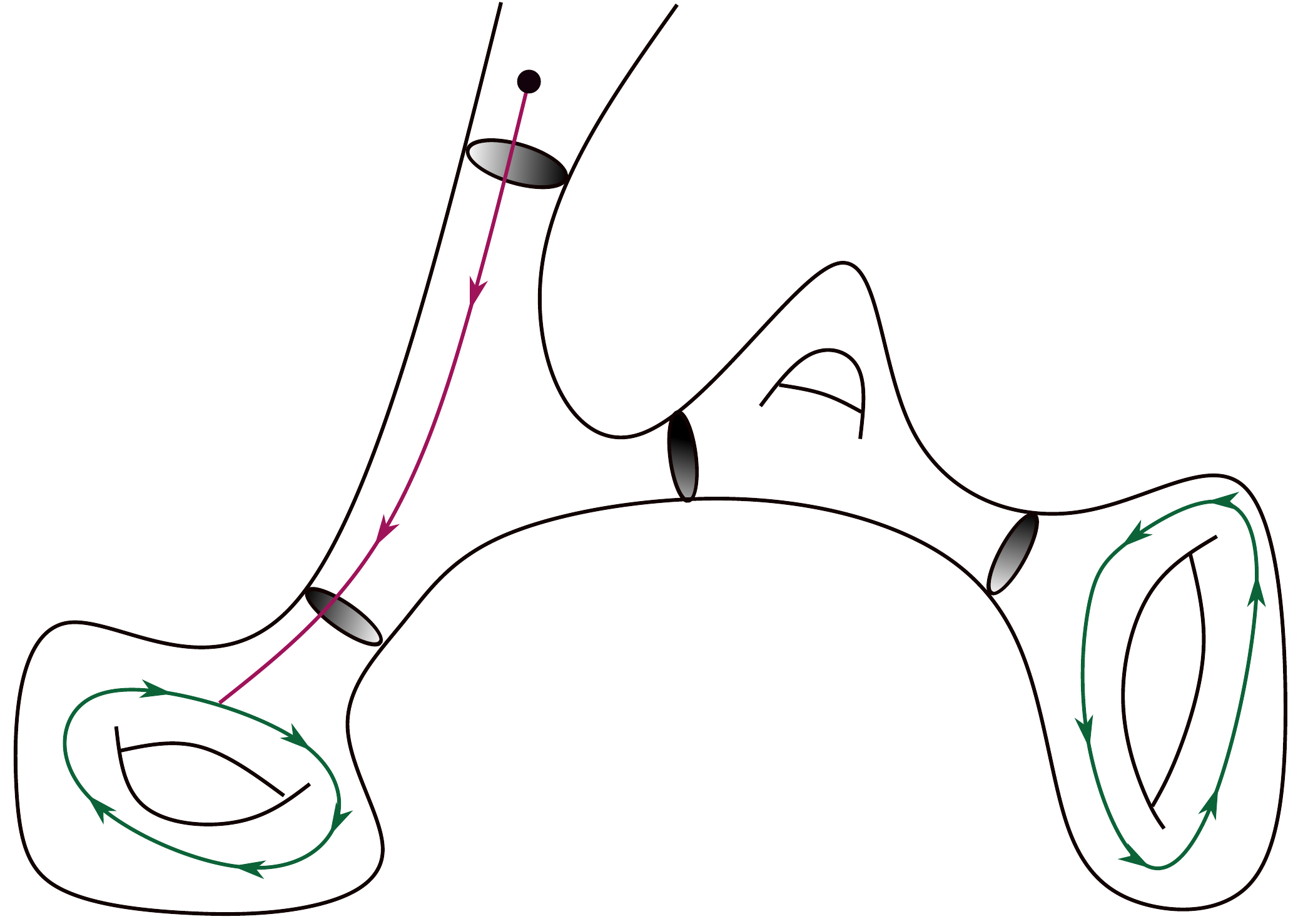}}%
    \put(0.39390093,0.67616628){\color[rgb]{0,0,0}\makebox(0,0)[lt]{\lineheight{1.25}\smash{\begin{tabular}[t]{l}{\tiny $\ast_\Sigma$}\end{tabular}}}}%
    \put(0.03902704,0.04628802){\color[rgb]{0,0,0}\makebox(0,0)[lt]{\lineheight{1.25}\smash{\begin{tabular}[t]{l}{\tiny $l'$}\end{tabular}}}}%
    \put(0.92890485,0.03970319){\color[rgb]{0,0,0}\makebox(0,0)[lt]{\lineheight{1.25}\smash{\begin{tabular}[t]{l}{\tiny $l''$}\end{tabular}}}}%
    \put(0.66569911,0.34556632){\color[rgb]{0,0,0}\makebox(0,0)[lt]{\lineheight{1.25}\smash{\begin{tabular}[t]{l}{\tiny $\beta$}\end{tabular}}}}%
    \put(0.36134918,0.4372345){\color[rgb]{0,0,0}\makebox(0,0)[lt]{\lineheight{1.25}\smash{\begin{tabular}[t]{l}{\tiny $\beta'$}\end{tabular}}}}%
    \put(0,0){\includegraphics[width=\unitlength,page=2]{kneser_no_positive_leaf_in_K_1.pdf}}%
  \end{picture}%
\endgroup%
  
\caption{$l',l''$ and the connecting arc $\beta$.}
\label{fig:no_positive_leaf_K_1}
\end{subfigure}
\caption{No positive leaves.}
\end{figure}

Now, both $h(\Sigma')$ and $h(\Sigma'')$ are in the 
component of $K\setminus \ast$ containing $K_1$ (Fig.\ \ref{fig:no_positive_leaf_K_1}).
Since $G$ is a tree, there exists
arc $\beta$ connecting $l'$ to $l''$
such that $\beta$ meets no
component of $h^{-1}(\ast)$ more than once.
The intersection $\beta \cap h^{-1}(\ast)$
separates $\beta$ into connected 
subarcs $\beta_0,\dots,\beta_{2m}$, $m> 0$.

Connect $\ast_\Sigma$
%the base point of $\Sigma$ 
to $l'$ 
via 
an arc 
%$\beta'
$\beta'\subset \Sigma$.
%, respectively
Then the image of the class $[\beta' l'\beta l''\beta^{-1}(\beta')^{-1}]$ under 
%$\phi$
\eqref{eq:homo_boundary_to_A_1}
has the form $y x'  y_0\cdots y_{2m} x''y_{2m}^{-1}\cdots y_0^{-1}  y^{-1}\in A_1$ 
with $m \geq 0$, for some $y\in A_1\ast A_2$ induced by $\beta'$,
where the elements $x'=[h\circ l']$ and $x''=[h\circ l'']$ are in $A_1$, 
and the element
$y_i$, induced from $\beta_i$, is in $A_j$, $i\equiv j$ (mod $2$).   
This implies some $h\circ \beta_i$, say $h\circ \beta_0$,
is null-homotopic. Performing the operation 
in the proof of Claim $1$ to merge 
the disks containing the endpoints of $\beta_0$ via $\beta_0$
(Fig.\ \ref{fig:homotopy}), we get a contradiction to the minimality. 
This completes the proof of Claim $2$.

Now, if $\mathbf{G}$ contains a non-degenerate tree, 
a tree having at least one edge, then by Claim $2$ 
it has at least one end with genus $0$. 
This implies that there exists
a disk $D$ in $h^{-1}(\ast)$ with $\partial D$
cutting off a disk $D'$ from $\surface$.  
Push $D\cup D'$ away from $\partial M$ to 
get a $2$-sphere $\Stwo$ in the interior of $M$. 
Then, since $\pi_2(K)$ is trivial, 
one can deform $h$ so that the disk $D$ in $h^{-1}(\ast)$ 
is replaced with the $2$-sphere $\Stwo$ without 
affecting other components of
$h^{-1}(\ast)$ (see \cite[p.\ 66]{Hem:04}), 
contradicting the minimality of $h$. Hence $\mathbf{G}$ is 
a collection of nodes without edges, that is, 
no disk in $h^{-1}(\ast)$ (i.e.\ $n_d=0$).

%%%%%%%%%%%%%%%%%%%%%claim 3 
 
\textbf{Claim $3$: There is no more than one $2$-sphere.}
This follows from the standard binding tie argument \cite{Sta:65},
\cite{Hei:71}, \cite[p.\ 67]{Hem:04}. 
For the sake of completeness, we outline its proof below.
If $h^{-1}(\ast)$ contains more than one $2$-sphere,
then we consider arcs with two ends lying 
in different components 
of $h^{-1}(\ast)$ and mapped to null-homotopic loops
under $h$.
Let $\alpha$ be such an arc that minimizes $\# h^{-1}(\ast)\cap \alpha$, 
and  $\Stwo_1$ and $\Stwo_2$ be the $2$-spheres in $h^{-1}(\ast)$ 
connected by $\alpha$. Then one can show that the interior 
of $\alpha$ must have trivial intersection with $h^{-1}(\ast)$
\cite[p.67]{Hem:04}, and thus one can homotopy $h$ so that 
$\Stwo_1$ and $\Stwo_2$ are
replaced by the $2$-sphere given by 
the union 
\[
\partial \mathfrak{N}(\alpha)\setminus (\Stwo_1\cup \Stwo_2)\bigcup (\Stwo_1\cup \Stwo_2)\setminus \mathring{\mathfrak{N}}(\alpha),
\]
with other components in $h^{-1}(\ast)$ intact,
where $\mathring{\mathfrak{N}}(\alpha)$ is the interior of 
a regular neighborhood $\mathfrak{N}(\alpha)$ 
of $\alpha$. This again contradicts the minimality of $\# h^{-1}(\ast)$.

%%%%%%%%%%%%%%%%%%%%%%%%%%%%%
%%%%%%%%%%%%%%%%%%%%%%%%%%%%%starting from here
As a result, $h^{-1}(\ast)$ contains only one $2$-sphere, 
which 
%separates $M$
%into two connected components; we denote their closures 
%by $\tilde{M}_1$ and $\tilde{M}_2$. 
%This 
yields a connected sum $M\simeq M_1\# M_2$,  
%where $M_i$ is obtained by capping off the spherical component in 
%$\partial\tilde{M}_i$ with a $3$-ball $B^3$.
%Furthermore, we may choose the base point $\ast_{M_i}$ 
%to be in a tubular neighborhood
%of $\Stwo_\phi$ such that $h(\ast_{M_i})=\frac{i}{3}\in I\subset K$. 
%Then there are natural arcs $\delta_i$ connecting $\ast_{M_i}$ to $\ast_{M}$ 
%in $\tilde{M}_i$
%and isomorphisms $\pi_1(M_i,\ast_{M_i})\simeq A_i$ such that
%the following diagram commutes   
and the restriction of $h$ on $M_i$ and an connecting arc $\delta_i\subset M_i$ induces an isomorphism $\phi_i:\pi_1(M_i)\rightarrow A_i$,
$i=1,2$, such that the diagram
%\begin{center}
\begin{equation}\label{diag:conn_from_M_i_to_M}
\begin{tikzpicture}[scale=.8, every node/.style={transform shape}, baseline = (current bounding box.center)]
\node (M) at (0,1.5) {$\pi_1(M)$};
\node (Mi) at (0,0) {$\pi_1(M_i)$};
\node (A)  at (5,1.5) {$A_1\ast A_2$};
\node (Ai) at (5,0) {$A_i$};

\draw [->] (M) to node [below]{$\sim$} node [above]{$\phi$} (A);
\draw [->] (Mi) to node [below]{$\sim$} node [above]{$\phi_i$}(Ai);
\draw [->] (Mi) to node [right]{$\delta_i$}(M);
\draw [->] (Ai) to (A);
\end{tikzpicture} 
\end{equation}
%\end{center} 
commutes.
%%%%%%%%%%%%%%%%%%%%%%%%%boundary

\textbf{Step $2$: Locate $\surface_1$ and $\surface_2$.}
Here we show that $\partial M_i=\surface_i$, $i=1,2$.
Suppose there is a component $\Sigma$ of 
$\surface_1$ in $\partial M_2$, 
that is, $h(\Sigma)$ in the component of $K\setminus \ast$
containing $K_2$.  
%As we shall see, this would lead to
%a contradiction as in the first part of the proof of Claim $2$: 
Then we choose a loop $l$ in $\Sigma$ that is essential
in $M$, 
%by \eqref{eq:non-trivial:exact_seq_for_Sigma_M}, 
and connect it to $\ast_M$ in $M_2$. Denote the resulting loop 
by $\tilde{l}$, and remark that $x=[h\circ \tilde{l}]$ is non-trivial in $A_2$. 
On the other hand, the composition \eqref{eq:homo_boundary_to_A_1}
implies $y x  y^{-1}\in A_1$ for some $y\in A_1\ast A_2$, but this is 
possible only when $x=1$. 
Thus, every component of $\surface_1$ must be 
in $\partial M_1$, and similarly, $\surface_2=\partial M_2$.

\textbf{Step $3$: Compatibility between $\phi_i,\phi_{ij}$.}
Step $1$ shows that {\bf III} in the diagram \eqref{diag:moving_delta_ij}
is commutative. Here we show that {\bf I} and {\bf II} 
in \eqref{diag:moving_delta_ij} can also be made commutative. 
\begin{equation}\label{diag:moving_delta_ij}
\begin{tikzpicture}[scale=.8, every node/.style={transform shape}, baseline = (current bounding box.center)]
\node (M) at (0,1.5) {$\pi_1(M)$};
\node (Mi) at (0,0) {$\pi_1(M_i )$};
\node (A)  at (4,1.5) {$A_1\ast A_2$};
\node (Ai) at (4,0) {$A_i$};
\node (S) at (0,-1.5) {$\pi_1(\Sigma_{ij} )$};

\node  at (2,.75) {\bf III};
\node   at (2,-.75) {\bf II};
\node  at (-.9,0) {\bf I};

\draw [->] (M) to node [below]{$\sim$}node[above]{$\phi$}(A);
\draw [->] (Mi) to node [below]{$\sim$}node[above]{$\phi_i$}(Ai);
\draw [->] (Mi) to node [right]{ $\delta_i$} (M);
\draw [->] (Ai) to node [right]{$\iota$}(A);
\draw [->] (S) to (Mi);
\draw [->] (S) to [out=180, in=180] node [left]{$\delta_{ij}$}(M);
\draw [->] (S) to [out=0,in=-90] node [above] {$\phi_{ij}$}(Ai);
\end{tikzpicture}
\end{equation}
Choose an arc $\epsilon_{ij}'$ connecting the base point 
of $\Sigma_{ij}$ to the base point of $M_i$. 
Then there exists an element $g\in A_1\ast A_2$ such that 
\[\phi\circ \delta_{ij}(x)=g\big(\phi\circ\delta_i\circ\epsilon_{ij}'(x)\big)g^{-1},
\forall x \in \pi_1(\Sigma_{ij}).\]
On the other hand, since $\phi\circ\delta_i=\iota\circ\phi_i$ and 
$\phi\circ \delta_{ij}$ factors through $\phi_{ij}$, 
both images of $\pi_1(\Sigma_{ij})$
under $\phi\circ \delta_{ij}$ and $\phi\circ\delta_i\circ\epsilon_{ij}'$
are in $A_i$. So, $g$ must be in $A_i$.
Choose a loop $\beta$ in $M_i$ representing $\phi_i^{-1}(g)$.
Then the induced homomorphism
\[\pi_1(\Sigma_{ij})\xrightarrow{\epsilon_{ij}} \pi_1(M_i)\]
of the new connecting arc $\epsilon_{ij}:=\epsilon_{ij}'\beta$ 
makes \eqref{diag:moving_delta_ij} commute.

Now, if $\partial M$ contains spherical components, 
we cap off those components by some $3$-balls
to obtain a new $3$-manifold $\widehat{M}$ with $\partial \widehat{M}$ containing no spherical components. Apply the above argument in \textbf{Step $1$} to $\widehat{M}$, isotoping the $2$-sphere $h^{-1}(\ast)$
away from those $3$-balls, to obtain the desired connected sum.
Note that \textbf{Step $2$} and \textbf{Step $3$} can go through without modification in this case.

%%%%not sure if the below diagram is needed.   
%%%%%%%%%%%%%%%%%%%%%%%%%%%%%%%%%%%%%%%%%%%%%%%
%%%%%%%%%%%%%%%%%%%%%%%%%%%%%%%%%%%%%%%%%%%%%%%
%%%%%%%%%%%%%%%%%%%%%%%%%%%%%%%%%%%%%%%%%%%%%%%%
\nada{ 
\begin{tikzpicture}[scale=.8, every node/.style={transform shape}, baseline = (current bounding box.center)]
\node (baseM) at (2.5,1.3) {$\ast_{M}$};
\node (baseMi) at (2,-.2) {$\ast_{M_i}$};
\node (baseij) at (0,0) {$\ast_{ij}$};
\draw [thick, ->] (baseij) to [out=60, in=-180] node [above]{$\delta_{ij}$}(baseM);
\draw [thick,->] (baseij) to [out=0, in=-170] node [below]{$\gamma$}(baseMi);
\draw [thick,->] (baseMi) to [looseness=5, out=210, in= 330] 
node[pos=0.2,below ] {$\beta$}(baseMi);
\draw [thick,->] (baseMi) to [out=80, in=-80] node [right]{$\delta_i$}(baseM);
\end{tikzpicture}
}
%%%%%%%%%%%%%%%%%%%%%%%%%%%%%%%%%%%%%%%%%%%%%%%
%%%%%%%%%%%%%%%%%%%%%%%%%%%%%%%%%%%%%%%%%%%%%%%
%%%%%%%%%%%%%%%%%%%%%%%%%%%%%%%%%%%%%%%%%%%%%%%%
\end{proof}

\begin{remark}\label{rmk:classic_Kneser_conjecture}
Here we explain how Theorem \ref{intro:teo:g_kneser_conjecture} implies
the classical Kneser conjecture \cite{Hei:71}, \cite[Chap.\ $7$]{Hem:04}, 
where $M$ is assumed to be $\partial$-irreducible.  
Consider the induced homomorphism $\pi_1(\Sigma)\xrightarrow{\delta} \pi_1(M)$, and suppose the composition
\[\pi_1(\Sigma)\xrightarrow{\delta} \pi_1(M)
\xrightarrow[\sim]{\phi} A_1\ast A_2,\]
does not factor through $A_1$ or $A_2$.
Then, since $M$ is $\partial$-irreducible, and  
$\pi_1(\Sigma)$ is indecomposable, 
$\pi_1(\Sigma)$ can be identified with
a subgroup in $g A_i g^{-1}$, $i=1$ or $2$, 
for some $g\in A_1\ast A_2$. 
Replace the connecting arc $\delta$
with $\delta\gamma$, 
where $\gamma$ is a loop in $M$ 
representing $\phi^{-1}(g^{-1})$. 
Then the composition
\[\pi_1(\Sigma)\xrightarrow{\delta\gamma} \pi_1(M)
\xrightarrow[\sim]{\phi} A_1\ast A_2\]
factors through $A_i$. Applying this
%construction 
to every component of $\partial M$,
we see that the assumption of the classical Kneser conjecture implies 
the condition of Theorem \ref{intro:teo:g_kneser_conjecture}. 
\end{remark}
%%%
%%%
%%%
%%%
%%%
In the proof of Theorem \ref{intro:main_thm}, 
Theorem \ref{intro:teo:g_kneser_conjecture} is used in the following form. 
\begin{corollary}\label{cor:application_g_kneser}
Given two $3$-manifolds $M, M'$ with no spherical boundary components, 
suppose $M\simeq M_1 \# M_2$, and
there exist isomorphisms 
\[\phi_M:\pi_1(M)\rightarrow \pi_1(M') \quad and 
\quad \phi_{ij}:\pi_1(\Sigma_{ij})\rightarrow \pi_1(\Sigma_{ij}')\] 
such that the following diagram 
\begin{equation}\label{diag:compatiability_MMprime}
\begin{tikzpicture}[scale=.8, every node/.style={transform shape}, baseline = (current bounding box.center)]
\node (M) at (0,2) {$\pi_1(M)$};
\node (Mprime) at (5,2) {$\pi_1(M')$};
\node (B)  at (0,0) {$\pi_1(\Sigma_{ij})$};
\node (Bprime) at (5,0) {$\pi_1(\Sigma_{ij}')$};

\draw [->] (M) to node [above]{$\phi_M$}(Mprime);
\draw [->] (B) to node [above]{$\phi_{ij}$}(Bprime);
\draw [->] (B) to   (M);
\draw [->] (Bprime) to  (Mprime);
\end{tikzpicture} 
\end{equation}  
commutes, up to conjugation,
for every component $\Sigma_{ij}\subset \partial M$, where
\[\bigcup_j \Sigma_{ij}=\partial M_i \quad \text{and} 
\quad \bigcup_{i,j} \Sigma_{ij}'=\partial M'.\] 
Then there exist $M_1'$, $M_2'$ such that 
$M'\simeq M_1'\# M_2'$, $\partial M_i'=\bigcup_j \Sigma_{ij}'$,
and isomorphisms 
\[\phi_i:\pi_1(M_i)\rightarrow \pi_1(M_i'),\quad i=1,2\]  
such that the following diagram
\begin{center}
\begin{tikzpicture}[scale=.8, every node/.style={transform shape}, baseline = (current bounding box.center)]
\node (Mi) at (0,2) {$\pi_1(M_i)$};
\node (Miprime) at (5,2) {$\pi_1(M_i')$};
\node (Bi)  at (0,0) {$\pi_1(\Sigma_{ij})$};
\node (Biprime) at (5,0) {$\pi_1(\Sigma_{ij}')$};

\draw [->] (Mi)to node [above]{$\phi_i$}(Miprime);
\draw [->] (Bi) to node [above]{$\phi_{ij}$}(Biprime);
\draw [->] (Bi) to  (Mi);
\draw [->] (Biprime) to (Miprime);
\end{tikzpicture} 
\end{center}
commutes, up to conjugation.
\end{corollary}
\begin{proof}
Modify connecting arcs between base points of $\Sigma_{ij}'$ and
$M'$ to make the diagram \eqref{diag:compatiability_MMprime}
commute strictly.
Then apply Theorem \ref{intro:teo:g_kneser_conjecture}.
\end{proof}

%%%%%%%%%%%%%%%%
%%%%%%%%%%%%%%%%%%
%%%%%%%%%%%%%%%%%%%%
%%%%%%%%%%%%%%%%%% 
%%%%%%%%%%%%%%%%%%%%%%

\section{Proof of the main theorem}\label{sec:proof}
%%The definition of foundamental tree associated to a non-connected scene
%%Splitting the tree:
%%1)Remove spherical parts
%%Lemma:restore the spherical parts
%%2)Split the tree into non-splittable part
%%Lemma restore non-splittable part
%%The main theorem
%%Use induction forward on the depth of a tree
%%____________________________

\begin{proof}[Proof of Theorem \ref{intro:main_thm}]
Suppose two pairs $\SSS=\pair$, $\SSS'=\pairprime$ have
equivalent fundamental trees. Then
after reindexing 
%$\Sigma_i'$, $F_j'$ 
if necessary,  
there are isomorphisms 
\begin{align*}
\phi_{\Sigma_i}:\pi_1(\Sigma_i)&\rightarrow \pi_1(\Sigma_i'),& i=1,\cdots,n\\
\phi_{F_j}:\pi_1(F_j)&\rightarrow \pi_1(F_j'), &j=0,\cdots,n,
\end{align*}
such that for every $\Sigma_i\subset F_j$, $\Sigma_i'\subset F_j'$ 
the diagram
\begin{equation}\label{diag:equivalent_fund_tree}
\begin{tikzpicture}[scale=.8, every node/.style={transform shape}, baseline = (current bounding box.center)]
\node (Lu) at (0,1.5){$\pi_1(F_j)$};
\node (Ll) at (0,0){$\pi_1(\Sigma_i)$};
\node (Ru) at (3,1.5){$\pi_1(F_j')$};
\node (Rl) at (3,0){$\pi_1(\Sigma_i')$};

\draw [->] (Lu) to node [above]{\footnotesize $\phi_{F_j}$}(Ru);
\draw [->] (Ll) to node [above]{\footnotesize $\phi_{\Sigma_i}$}(Rl);
\draw [->] (Ll) to  (Lu); 
\draw [->] (Rl) to  (Ru); 
\end{tikzpicture}
\end{equation}
commutes, up to conjugation.

We consider the special case 
where $\SSS$ is non-splittable first, and then reduce 
the general case to the special 
one via non-splittable 
graft decomposition 
%of $\SSS$ 
and Theorem \ref{intro:teo:g_kneser_conjecture}.

\textbf{Case $1$. Non-splittable pair $\SSS=\pair$.}
First, observe that $\SSS$ is non-splittable if and only if
every solid part of $\SSS$ is irreducible.
By the diagram \eqref{diag:equivalent_fund_tree} 
and Theorem \ref{intro:teo:g_kneser_conjecture}, 
solid parts of $\pairprime$ are irreducible,
and hence $\SSS'$ is also non-splittable.

Secondly, by the Dehn-Nielsen-Baer theorem, 
$\phi_{\Sigma_i}$ and the orientation of $\Sigma_i$,
$i=1,\cdots,n$,
induces a homeomorphism $f_j$ from $\partial F_j$ to $\partial F_j'$,
for each $j$.
This homeomorphism extends to a map
$g_j$ inducing $\phi_{F_j}$
by the construction in \cite[Lemma $13.8$]{Hem:04}.
Since $\phi_{F_j}$ is an isomorphism 
and $g_j\vert_{\partial F_j}=f_j$ is a homeomorphism,
by \cite[Theorem $13.6$]{Hem:04}, there is a 
homeomorphism $h_j$ homotopic to $g_j$ relative to
the boundary.   
Gluing $h_j$, $j=0,\cdot,n$, together along $\partial F_j$,
we obtain an equivalence between $\SSS$ and $\SSS'$.

%%%%%%%%%%%%%%%%%%%%%%%%%%%%%%%%%%%%%%%%%%%%%%%%%%%%%%%%%%%%%%%%%%%
%%%%%%%%%%%%%%%%%%%%%%%%%%%%%%%%%%%%%%%%%%%%%%%%%%%%%%%%%%%%%%%%%%%
%%%%%%%%%%%%%%%%%%%%%%%%%%%%%%%%%%%%%%%%%%%%%%%%%%%%%%%%%%%%%%%%%%%
%%%%%%%%%%%%%%%%%%%%%%%%%%%%%%%%%%%%%%%%%%%%%%%%%%%%%%%%%%%%%%%%%%%
%%%%%%%%%%%%%%%%%%%%%%%%%%%%%%%%%%%%%%%%%%%%%%%%%%%%%%%%%%%%%%%%%%%
%%%%%%%%%%%%%%%%%%%%%%%%%%%%%%%%%%%%%%%%%%%%%%%%%%%%%%%%%%%%%%%%%%%
%%%%%%%%%%%%%%%%%%%%%%%%%%%%%%%%%%%%%%%%%%%%%%%%%%%%%%%%%%%%%%%%%%%
%%%%%%%%%%%%%%%%%%%%%%%%%%%%%%%%%%%%%%%%%%%%%%%%%%%%%%%%%%%%%%%%%%%
%
\textbf{Case $2$. General pair $\SSS=\pair$.} 
By Proposition \ref{prop:nonsplit_graft_decomp}, 
the pair $\SSS$ admits a unique non-splittable 
graft decomposition:
\begin{equation}\label{eq:graft_decomp_of_S}
\SSS\simeq \SSS_1\dashleftarrow \SSS_2\dashleftarrow\dots \dashleftarrow\SSS_m.
\end{equation}
%
%%the discussion about the base point and connecting arcs omitted
 
%
Properly choosing base-points and connecting arcs between them,
we obtain that \eqref{eq:graft_decomp_of_S} induces
a graft decomposition of the fundamental tree of $\SSS$: 
\begin{equation}\label{eq:induced_decomp_S}
\mathcal{FT}(\SSS)\simeq \mathcal{FT}(\SSS_1)\dashleftarrow\mathcal{FT}(\SSS_2)\dashleftarrow\dots\dashleftarrow \mathcal{FT}(\SSS_m).
\end{equation}  
Since $\mathcal{FT}(\SSS)$ and $\mathcal{FT}(\SSS')$
are equivalent, \eqref{eq:induced_decomp_S} induces a graft decomposition 
of $\mathcal{FT}(\SSS')$:
\begin{equation}\label{eq:induced_decomp_Sprime}
\mathcal{FT}(\SSS')\simeq \mathcal{T}_1\dashleftarrow\mathcal{T}_2\dashleftarrow\dots\dashleftarrow \mathcal{T}_m.
\end{equation} 

\textbf{$\bullet$ The algebraic graft decomposition \eqref{eq:induced_decomp_Sprime} can be realized topologically.}
In other words, we want to show that,
after reindexing if necessary, 
the non-splittable graft decomposition of $\SSS'$:  
\[\SSS'\simeq  \SSS_1'\dashleftarrow \SSS_2'\dashleftarrow\dots\dashleftarrow  \SSS_m'\]
satisfies $\mathcal{FT}(\SSS'_i)\simeq \mathcal{T}_i$, $i=1,\dots, m$.
%%%%%%%%%%%%%%%%%%%%%%%%%%From here tomorrow

We prove the claim by induction on the length $m$ 
of the graft decomposition of $\SSS$.
When $m=1$, $\SSS$ and therefore $\SSS'$ are non-splittable 
by Theorem \ref{intro:teo:g_kneser_conjecture}, and 
it is the special case. 

%%%Depth tree index
Suppose $m>1$. We index solid parts $F_j$, $j=0,\dots,n$, of
$\SSS$ such that $j>i$ if 
$F_j$ has a greater depth than $F_i$ 
in $\Lambda_\SSS$. Consider   
\begin{equation}\label{eq:max_non_splittable_comp} 
k:=\operatorname{max}\{j\mid F_j \text{ is not irreducible }\}.
\end{equation}

\textbf{Subcase $1$:} If $\partial F_k$ contains at least one 
spherical component $\Stwo$ that does not separate $F_k$ from $\infty$,  
then $\Stwo$ bounds a $3$-ball $B$ with $\mathring{B}\cap \surface=\emptyset$. 
Up to reindinxing \eqref{eq:induced_decomp_S}, we may assume
$\SSS_m$ is the trivial pair of genus $0$ containing $\Stwo$. 
The corresponding component $\Stwo'$ of $\surface'$ 
is also a $2$-sphere, and $\SSS_m'=(\sphere,\Stwo')$ 
realizes $\mathcal{T}_m$. 
Let $\tilde{\SSS}'$ be the pair obtained by removing $\Stwo'$ 
from $\surface'$. 
Then we have the graft decomposition of $\mathcal{FT}(\tilde{\SSS}')$:
\begin{equation}\label{eq:graft_decomp_tildeS_1}
\mathcal{FT}(\tilde{\SSS}')\simeq \mathcal{T}_1\dashleftarrow\dots\dashleftarrow\mathcal{T}_{m-1}. 
\end{equation}      
By the induction hypothesis, the non-splittabe graft decomposition of 
$\tilde{\SSS'}$ realizes \eqref{eq:graft_decomp_tildeS_1};
the claim then follows from the fact that $\SSS\simeq \tilde{\SSS'}\dashleftarrow \SSS_m'$.

\textbf{Subcase $2$:} Suppose $\partial F_k$ contains 
only one spherical component $\Stwo$ 
which separates $F_k$ from $\infty$.
Let $\SSS_l$ be the trivial pair of genus $0$ 
induced from $\Stwo$. Then 
\[\SSS\simeq \tilde{\SSS}\dashleftarrow \SSS_l\dashleftarrow\overline{\SSS},\] 
where $\tilde{\SSS}$ is the pair obtained by 
removing components of $\surface$ on the side of 
$\Stwo$ opposite to $\infty$, and $\overline{\SSS}$ is the pair 
obtained by removing components of $\surface$ 
on the same side of $\Stwo$ as $\infty$.
This implies, up to reindexing, the graft decompositions: 
\begin{align*}
\mathcal{FT}(\tilde{\SSS})&\simeq \mathcal{FT}(\SSS_1)\dashleftarrow\dots\dashleftarrow
\mathcal{FT}(\SSS_{l-1})\\
\mathcal{FT}(\overline{\SSS})&\simeq \mathcal{FT}(\SSS_{l+1})
\dashleftarrow\dots\dashleftarrow \mathcal{FT}(\SSS_m).
\end{align*}

On the other hand, the corresponding $2$-sphere $\Stwo'$ in $\surface'$
realizes $\mathcal{T}_l$ and induces a graft decomposition of $\SSS'$:
\[\SSS'\simeq \tilde{\SSS}^{'}\dashleftarrow \SSS_l'\dashleftarrow \overline{\SSS}'\]
\begin{align}
&\mathcal{FT}(\tilde{\SSS}^{'})\simeq \mathcal{T}_1\dashleftarrow\dots\dashleftarrow
\mathcal{T}_{l-1}\label{eq:graft_decomp_tildeS_2}\\
\text{such that} \qquad &\mathcal{FT}(\SSS_l')\simeq \mathcal{T}_l\nonumber\\
&\mathcal{FT}(\overline{\SSS}')\simeq \mathcal{T}_{l+1}\dashleftarrow\dots\dashleftarrow \mathcal{T}_m.\label{eq:graft_decomp_overS}
\end{align}
As with Subcase $1$, the induction hypothesis implies
the claim.
  
\textbf{Subcase $3$:} Suppose $\partial F_k$ has no spherical components.
Then there exists a $2$-sphere $\Stwo\subset F_k$ inducing
a connected sum decomposition $F_k\simeq \tilde{F}_k\# \overline{F}_k$ 
such that the one, say $\overline{F}_k$, containing 
the side of $\Stwo$ 
opposite to $\infty$ is irreducible. Thus removing components of $\surface$
that are on the same side of $\Stwo$ as $\infty$ gives a non-splittable
pair, and up to reindexing, we may assume it is $\SSS_m$
in \eqref{eq:graft_decomp_of_S}.
 
Let $B$ (resp.\ $A_m$) denote the fundamental group
$\pi_1(\tilde{F}_k)$ (resp.\ $\pi_1(\overline{F}_k)$).
Then $\pi_1(F)$ is isomorphic to a non-trivial
free product $B\ast A_m$ such that, for any component
$\Sigma$ of $\partial \tilde{F}_k$ (resp.\ of $\partial \overline{F}_k$), 
the homomorphism
\[\pi_1(\Sigma)\rightarrow \pi_1(F)\]
%\[\pi_1(\Sigma)\rightarrow \pi_1(\tilde{F_k})\quad \text{( \textbf{resp.} }
%\pi_1(\Sigma)\rightarrow \pi_1(\overline{F_k})\hspace*{.3em})\]
factors through $B$ (resp.\ $A_m$). 
Via equivalence \eqref{diag:equivalent_fund_tree} 
between $\SSS$ and $\SSS'$, 
the fundamental group of the solid part  
$F_k'$ of $\SSS'$ corresponding to $F_k$ also satisfies
\[\pi_1(F_k')\simeq B\ast A_m,\]
and for any $\pi_1(\Sigma')$ in $\mathcal{T}_m$ (resp.\ in 
$\mathcal{T}_{1}\dashleftarrow\dots\dashleftarrow\mathcal{T}_{m-1}$),
the homomorphism  
\[\pi_1(\Sigma')\rightarrow \pi_1(F_k')\simeq B\ast A_m\]
factors through $A_m$ (resp.\ $B$). 
Applying Corollary \ref{cor:application_g_kneser}, we obtain a decomposition
\[F_k'\simeq \tilde{F}_k'\# \overline{F}_k' \quad\text{with}\quad  
\pi_1(\tilde{F}_k')\simeq B,\quad \pi_1(\overline{F}_k')\simeq A_m\]
such that $\Sigma'\subset \partial F_k'$ is in $\partial \tilde{F}_k'$ 
(resp.\ $\partial\overline{F}_k'$) if 
the homomorphism 
\[\pi_1(\Sigma')\rightarrow \pi_1(F_k')\]
factors through
\[\pi_1(\Sigma')\rightarrow \pi_1(\tilde{F}_k')\quad(\text{ resp. }
\pi_1(\Sigma')\rightarrow \pi_1(\overline{F}_k')\hspace*{.3em}).\]
This implies a graft decomposition
\begin{equation*}
\SSS'\simeq \tilde{\SSS}'\dashleftarrow \SSS_m'
\quad with \quad 
\mathcal{FT}(\tilde{\SSS}')\simeq \mathcal{T}_1 \dashleftarrow\dots\dashleftarrow
\mathcal{T}_{m-1} \quad and\quad 
\mathcal{FT}(\SSS_m')\simeq  \mathcal{T}_m,
\end{equation*}
where $\tilde{\SSS}'$ is the pair consisting of components
of $\surface$ that are not in $\partial \overline{F}_k'$,
and $\SSS_m'$ is the pair containing only 
components of $\partial \overline{F}_k'$. 
By Theorem \ref{intro:teo:g_kneser_conjecture}, 
$\SSS_m'$ is necessarily non-splittable, and 
the claim follows from the induction hypothesis.
 
In this way, any equivalence between
$\mathcal{FT}(\SSS)$ and $\mathcal{FT}(\SSS')$ induces an equivalence
between $\mathcal{FT}(\SSS_i)$ and $\mathcal{FT}(\SSS_i')$
for every $i$. The problem is thus reduced to Case $1$, and     
once an equivalence between $\SSS_i$ and $\SSS_i'$ 
is established for every $i$, 
we glue them together to get an equivalence
between $\SSS$ and $\SSS'$.
\end{proof}

%%add rational homology 3-sphere here or in the intro!?
 
It is not difficult to see that Theorem \ref{intro:main_thm} 
can be generalized to surfaces in an irreducible, oriented $3$-manifold $M$
with a base point $\infty$.
\begin{definition}
We denote by a pair $\mathcal{M}=(M,\surface)$ a surface
$\surface$ in $M\setminus \infty$ with each component of $\surface$ 
separating $M$. 
Two pairs $(M,\surface), (M,\surface')$ are equivalent if 
there exists an orientation-preserving homeomorphism
fixing $\infty$ and sending $\surface$ to $\surface'$.
Denote the set of equivalence classes of $\mathcal{M}$ by
$\mathsf{Sur}_{M}$.
\end{definition}
As with the case of $\sphere$, one 
can define a mapping $\mathcal{FT}$ from 
$\mathsf{Sur}_{M}$ to $\mathsf{Grp}_{f,p}^{\op{BD}}$. 
Replacing $\sphere$ with $M$ in Definitions \ref{def:splitting_sphere} and \ref{def:splitting_sphere} allows us to define 
the graft decomposition of $\mathcal{M}$. The proof of 
Proposition \ref{prop:nonsplit_graft_decomp} can be carried over to show 
the existence and uniqueness of the graft decomposition
of $\mathcal{M}$. Thus, by
the generalized Kneser conjecture and Waldhausen theorem,
as done in the preceding proof, we obtain the following complete invariant.
 
\begin{corollary}
Two pairs $\mathcal{M},\mathcal{M}'$ are equivalent if and only if 
$\mathcal{FT}(\mathcal{M})$ and $\mathcal{FT}(\mathcal{M}')$
are equivalent.
\end{corollary}

%%Here we disucss nonbased case
%%A coloring on a graph
%%A tree always admits a coloring
%%barycentric subdivision of a colored tree
%%Induced intersection form of a colored diagram of groups
%%Equivalence, there exists colorings on 
%%the diagrams of groups and an equivalence between them.

%\section{Surface links}\label{sec:links}
%\input{links}
\section{Computation and examples}\label{sec:examples}
%%move examples and computation here
Here we compute invariants of handlebody links and their tunnels 
derived from the fundamental tree. 
The computation is done via the program Appcontour \cite{appcontour}.
%via homomorphisms of 
%$\pi_1(F_0)$ to a finite group $G$.  

\subsection{Example $1$}
\begin{definition}
Let $\SSS = \pair$ be the pair induced by 
a handlebody link $\HL$. Denote by $\Sigma_i$, $i=1,\dots,n$,  
components of $\surface$, and by $\mathcal{H}(\SSS)$ 
the set of homomorphisms from $\pi_1(F_0)$ to
$G$, up to automorphism of $G$.
%
%\draftMMM{aren't they *inner* automorphisms? Also note that appcontour command "ks\_A6" without
%option --inner would count up to automorphisms $x \mapsto x^h$ with $h \in S_6$}

An element $x$ in $\mathcal{H}(\SSS)$ 
is called proper with respect to $\Sigma_i$ 
if representing homomorphisms of $x$ are surjective, but become 
non-surjective after precomposing with the homomorphism
\[\pi_1(\Sigma_i)\rightarrow \pi_1(F_0).\]

An element $x$ in $\mathcal{H}(\SSS)$ 
is proper with respect to a subset $\mathcal{A}$ of $\{\Sigma_i\}_{i=1}^{n}$ 
if $x$ is proper with respect to every member in $\mathcal{A}$. 
The set of proper elements in $\mathcal{H}(\SSS)$ with respect to $\mathcal{A}$
is denoted by $\mathcal{PH}(\SSS)_{\mathcal{A}}$. 
\end{definition}

\begin{definition}[\textbf{$G$-images}]
Given $\mathcal{A}=\{\Sigma_{i_1},\cdots,\Sigma_{i_k}\}\subset \{\Sigma_i\}_{i=1}^{n}$,
the $G$-image of a handlebody link 
$\SSS= \pair$ with respect to $\mathcal{A}$ 
is a set of unordered $k$-tuples of 
subgroups of $G$, up to automorphism, 
parameterized by elements in $\mathcal{PH}(\SSS)_{\mathcal{A}}$ defined as follows:
\begin{equation}
G\text{-}\operatorname{im}(\SSS)_{\mathcal{A}}:=\{(H_{i_1},H_{i_2},\cdots, H_{i_k})_x\mid x\in \mathcal{PH}(\SSS)_{\mathcal{A}} \},
\end{equation}  
where $H_{i}$ in a $k$-tuple $(\cdots)_x$ 
is the image of the homomorphism
\[\operatorname{Ker}\Big(
\pi_1(\Sigma_{i})\xrightarrow{\iota_{i}} \pi_1(F_{i})\Big)< \pi_1(\Sigma_{i})
\rightarrow \pi_1(F_0)\xrightarrow{\phi} G,\]
and $\phi$ is a representative of $x$. 
The $k$-fold $G$-image of $\SSS$ is defined by  
\begin{equation}
G\text{-}\operatorname{im}(\SSS)^{k}:=\{G\text{-}\operatorname{im}(\SSS)_{\mathcal{A}} 
\mid \mathcal{A}\subset \{\Sigma_i\}_{i=1}^{n}, \vert \mathcal{A}\vert=k \}.
\end{equation} 
\end{definition}
When $k=1$, we call it the individual $G$-image of $\SSS$, and drop 
the superscript $1$. Corollary \ref{intro:invariant_handlebody_links} 
implies the $k$-fold $G$-image is an invariant of handlebody links.

We compute individual $G$-images of handlebody links
in Fig.\ \ref{fig:hl_family}. 
These handlebody links 
have the same linking number \cite{Miz:13}, and 
each of their components is a trivial handlebody knot. In addition,
they have homeomorphic complements as
$\mathrm{HL}2, \mathrm{HL}3$ in Fig.\ \ref{fig:hl_2_hl_3} can 
be obtained by twisting $\mathrm{HL}1$ along  
once-punctured annuli $A_2,A_3$ in Fig.\ \ref{fig:hl_1_annuli}, respectively 
(see \cite{LeeLee:12}, \cite{Mot:90}, or \cite[Sec.\ $4$]{BePaWa:19}
for the twist construction). Their individual $A_4$-images show that 
they are inequivalent.
 
\begin{figure}[ht]
\begin{subfigure}{0.48\textwidth}
\def\svgwidth{0.95\columnwidth}
%% Accompanies image file 'HL_1_annuli.pdf' (pdf, eps, ps)
\begingroup%
  \makeatletter%
  \providecommand\color[2][]{%
    \errmessage{(Inkscape) Color is used for the text in Inkscape, but the package 'color.sty' is not loaded}%
    \renewcommand\color[2][]{}%
  }%
  \providecommand\transparent[1]{%
    \errmessage{(Inkscape) Transparency is used (non-zero) for the text in Inkscape, but the package 'transparent.sty' is not loaded}%
    \renewcommand\transparent[1]{}%
  }%
  \providecommand\rotatebox[2]{#2}%
  \newcommand*\fsize{\dimexpr\f@size pt\relax}%
  \newcommand*\lineheight[1]{\fontsize{\fsize}{#1\fsize}\selectfont}%
  \ifx\svgwidth\undefined%
    \setlength{\unitlength}{1417.32283465bp}%
    \ifx\svgscale\undefined%
      \relax%
    \else%
      \setlength{\unitlength}{\unitlength * \real{\svgscale}}%
    \fi%
  \else%
    \setlength{\unitlength}{\svgwidth}%
  \fi%
  \global\let\svgwidth\undefined%
  \global\let\svgscale\undefined%
  \makeatother%
  \begin{picture}(1,0.6)%
    \lineheight{1}%
    \setlength\tabcolsep{0pt}%
    \put(0,0){\includegraphics[width=\unitlength,page=1]{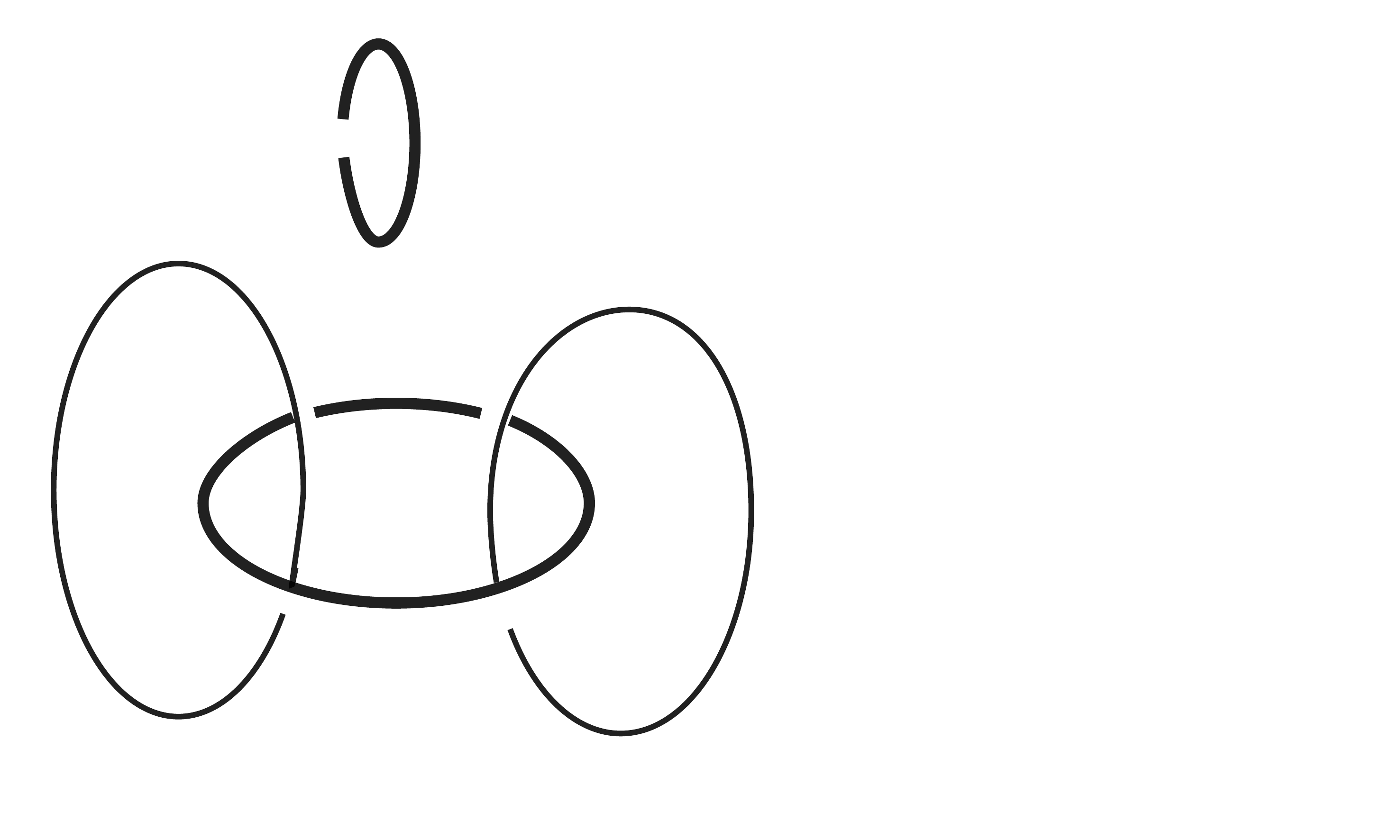}}%
    \put(0.18595079,0.01365491){\color[rgb]{0,0,0}\makebox(0,0)[lt]{\lineheight{1.25}\smash{\begin{tabular}[t]{l}{\small $\mathrm{HL}1$}\end{tabular}}}}%
    \put(0.25895309,0.19436038){\color[rgb]{0,0,0}\makebox(0,0)[lt]{\lineheight{1.25}\smash{\begin{tabular}[t]{l}{\tiny $\Sigma_1$}\end{tabular}}}}%
    \put(0.01285223,0.42459937){\color[rgb]{0,0,0}\makebox(0,0)[lt]{\lineheight{1.25}\smash{\begin{tabular}[t]{l}{\tiny $\Sigma_2$}\end{tabular}}}}%
    \put(0,0){\includegraphics[width=\unitlength,page=2]{HL_1_annuli.pdf}}%
    \put(0.91191815,0.2004097){\color[rgb]{0,0,0}\makebox(0,0)[lt]{\lineheight{1.25}\smash{\begin{tabular}[t]{l}{\tiny $A_3$}\end{tabular}}}}%
    \put(0,0){\includegraphics[width=\unitlength,page=3]{HL_1_annuli.pdf}}%
    \put(0.73659174,0.29675745){\color[rgb]{0,0,0}\makebox(0,0)[lt]{\lineheight{1.25}\smash{\begin{tabular}[t]{l}{\tiny $A_2$}\end{tabular}}}}%
    \put(0,0){\includegraphics[width=\unitlength,page=4]{HL_1_annuli.pdf}}%
  \end{picture}%
\endgroup%
 
\caption{$\op{HL}_1$ and one-punctured annuli.}
\label{fig:hl_1_annuli}
\end{subfigure}
\begin{subfigure}{0.48\textwidth}
\def\svgwidth{0.95\columnwidth}
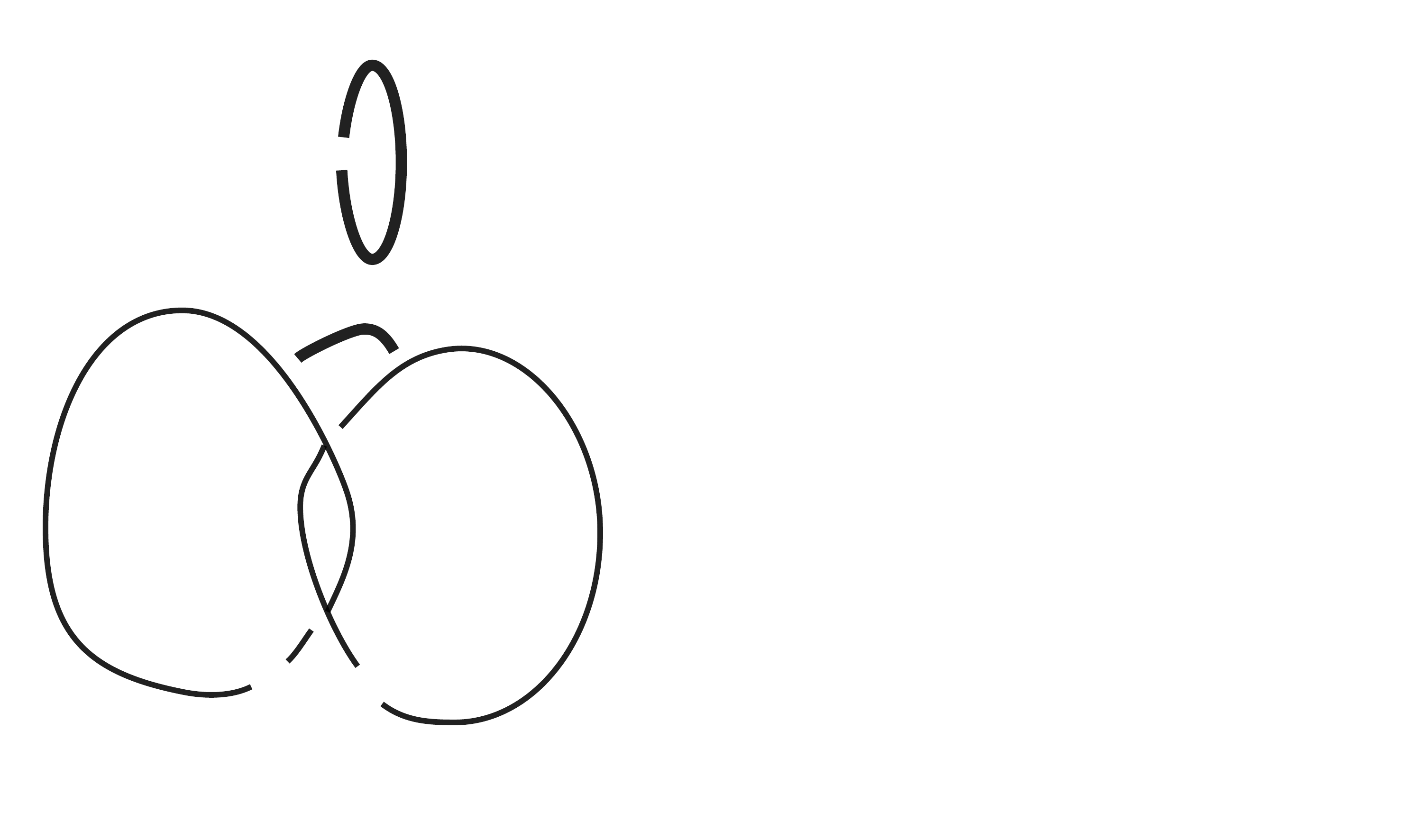 
\caption{Twisted $\op{HL}_1$.}
\label{fig:hl_2_hl_3}
\end{subfigure}
\caption{Handlebody link family.}
\label{fig:hl_family}
\end{figure}
$\Sigma_1$ in Fig.\ \ref{fig:hl_family} denotes the component 
in bold and $\Sigma_2$ the other component.  
As recorded in Table \ref{tab:ind_A_4_images_1},
there are $33$ proper homomorphisms with respect to 
$\Sigma_1$ and $33$ with respect to $\Sigma_2$. 
$18$ among them have that the image of $\pi_1(\Sigma_i)$ 
is $\mathbb{Z}_2\times \mathbb{Z}_2$ in $A_4$,
$12$ among them $\mathbb{Z}_3$, and $3$ among them 
$\mathbb{Z}_2$. 
%Their individual $A_4$-images 
%are recorded  
From Table \ref{tab:ind_A_4_images_1}, we also 
see that no ambient isotopy
can swap the two components of $\operatorname{HL}1$ or of $\operatorname{HL}3$.

\begin{table}[ht]
\caption{The $A_4$-images.}
\begin{center}
\begin{tabular}{c|c|r|r|r}
    \hline
    \multicolumn{2}{c|}{}& \multicolumn{3}{c}{Image of $\pi_1(\Sigma_i)$ in $A_4$}\\
    \cline{3-5}
    \multicolumn{2}{c|}{}&$\mathbb{Z}_2 \times \mathbb{Z}_2 :18$&$\mathbb{Z}_3 :12$&$\mathbb{Z}_2 :3$\\
    \hline
    \multirow{2}{*}{$\mathrm{HL}1$}&$\Sigma_1$ & 
              $\mathbb{Z}_2 \times \mathbb{Z}_2 :8$ & 
               $\mathbb{Z}_3 :9$ & $\mathbb{Z}_2 :2$\\
              &&$\mathbb{Z}_2 :10$& $\mathbf{0}:3$ & $\mathbf{0}:1$\\
    \cline{2-5}
               &$\Sigma_2$ & 
               $\mathbb{Z}_2 \times \mathbb{Z}_2 :8$ & 
               $\mathbb{Z}_3 :12$ & $\mathbb{Z}_2 :3$\\
              &&$\mathbb{Z}_2 :9$&&  \\
              &&$\mathbf{0}:1$&&  \\      
    \hline
    \multirow{2}{*}{$\mathrm{HL}2$}& $\Sigma_1$ & 
               $\mathbb{Z}_2 \times \mathbb{Z}_2 :8$ & 
               $\mathbb{Z}_3 :12$ & $\mathbb{Z}_2 :3$\\
              &&$\mathbb{Z}_2 :9$&&  \\
              &&$\mathbf{0}:1$&&  \\
    \cline{2-5}
               &$\Sigma_2$ & 
               $\mathbb{Z}_2 \times \mathbb{Z}_2 :8$ & 
               $\mathbb{Z}_3 :12$ & $\mathbb{Z}_2 :3$\\
              &&$\mathbb{Z}_2 :9$&&  \\
              &&$\mathbf{0}:1$&&  \\      
    \hline
    \multirow{2}{*}{$\mathrm{HL}3$}  &$\Sigma_1$ & 
              $\mathbb{Z}_2 \times \mathbb{Z}_2 :8$ & 
               $\mathbb{Z}_3 :12$ & $\mathbb{Z}_2 :3$\\
              &&$\mathbb{Z}_2 :9$&&  \\
              &&$\mathbf{0}:1$&&  \\  
                 
     \cline{2-5}  
              &$\Sigma_2$ & 
              $\mathbb{Z}_2 \times \mathbb{Z}_2 :8$ & 
               $\mathbb{Z}_3 :12$ & $\mathbb{Z}_2 :2$\\
              &&$\mathbb{Z}_2 :10$&   & $\mathbf{0}:1$\\

    \hline
\end{tabular}
\end{center}
\label{tab:ind_A_4_images_1}
\end{table}

On the other hand, there are only three proper homomorphisms with respect to 
$\{\Sigma_1,\Sigma_2\}$, 
and the $2$-fold $A_4$ images of $\operatorname{HL}1$, 
$\operatorname{HL}2$ and $\operatorname{HL}3$ are recorded in Table \ref{tab:2_fold_A_4_images_1}, which shows that the $2$-fold $A_4$-image is unable to
differentiate $\operatorname{HL}1$ and $\operatorname{HL}3$. 
 
\begin{table}[ht]
\caption{The $2$-fold $A_4$-images.}
\begin{center}
\begin{tabular}{c|c}
    & $(\Sigma_1,\Sigma_2)$\\ 
    \hline
    $\operatorname{HL}1$& $\{(\mathbb{Z}_2,\mathbb{Z}_3),(\mathbb{Z}_3,\mathbb{Z}_2\times \mathbb{Z}_2),(\mathbb{Z}_3,\mathbb{Z}_3)\}$\\
    \hline
    $\operatorname{HL}2$&$\{(\mathbb{Z}_2 \times \mathbb{Z}_2 ,\mathbb{Z}_3 ),(\mathbb{Z}_3,\mathbb{Z}_2 \times \mathbb{Z}_2),(\mathbb{Z}_3 ,\mathbb{Z}_3)\}$\\
    \hline   
    $\operatorname{HL}3$&$\{(\mathbb{Z}_3 ,\mathbb{Z}_2),(\mathbb{Z}_2 \times \mathbb{Z}_2 ,\mathbb{Z}_3 ),(\mathbb{Z}_3 ,\mathbb{Z}_3)\}$\\    
    \hline
\end{tabular}
\end{center}
\label{tab:2_fold_A_4_images_1}
\end{table}
%\end{example}

\subsection{Example 2}
Let $\gamma$ be a tunnel of a handlebody knot $\HK$.
We denote by $M_p(\gamma)$ the CW complex obtained by
attaching a $2$-disk $D$ 
to the complement $\Compl{\HK\cup\mathfrak{N}(\gamma)}$
via a degree $p$ map $\partial D\rightarrow C_\gamma$, 
where $C_\gamma$ is 
the boundary of a dual disk of $\gamma$ in
a regular neighborhood $\mathfrak{N}(\gamma)\subset \Compl\HK$.
In particular, if two tunnels $\gamma,\gamma'$
are equivalent, there is a homeomorphism between
$M_p(\gamma),M_p(\gamma')$. Denote by
$\vert\mathcal{H}(X)\vert$ the number of homomorphisms
from $\pi_1(X)$ to $A_5$, up to automorphism.

\begin{figure}[h]
\def\svgwidth{0.35\columnwidth}
%% Accompanies image file 'two_tunnels_hk6_11.pdf' (pdf, eps, ps)
\begingroup%
  \makeatletter%
  \providecommand\color[2][]{%
    \errmessage{(Inkscape) Color is used for the text in Inkscape, but the package 'color.sty' is not loaded}%
    \renewcommand\color[2][]{}%
  }%
  \providecommand\transparent[1]{%
    \errmessage{(Inkscape) Transparency is used (non-zero) for the text in Inkscape, but the package 'transparent.sty' is not loaded}%
    \renewcommand\transparent[1]{}%
  }%
  \providecommand\rotatebox[2]{#2}%
  \newcommand*\fsize{\dimexpr\f@size pt\relax}%
  \newcommand*\lineheight[1]{\fontsize{\fsize}{#1\fsize}\selectfont}%
  \ifx\svgwidth\undefined%
    \setlength{\unitlength}{1105.51181102bp}%
    \ifx\svgscale\undefined%
      \relax%
    \else%
      \setlength{\unitlength}{\unitlength * \real{\svgscale}}%
    \fi%
  \else%
    \setlength{\unitlength}{\svgwidth}%
  \fi%
  \global\let\svgwidth\undefined%
  \global\let\svgscale\undefined%
  \makeatother%
  \begin{picture}(1,0.79487179)%
    \lineheight{1}%
    \setlength\tabcolsep{0pt}%
    \put(0,0){\includegraphics[width=\unitlength,page=1]{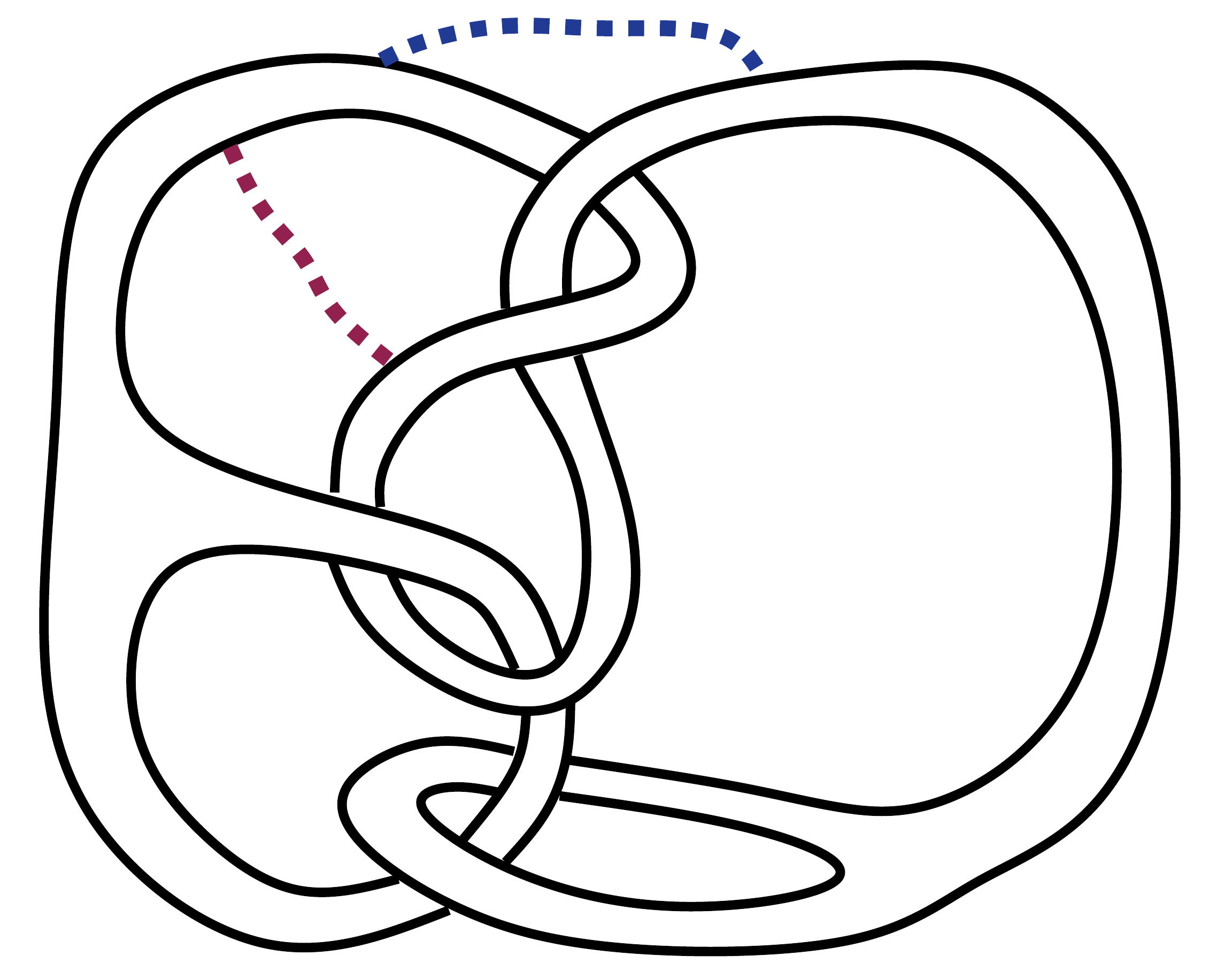}}%
    \put(0.25608444,0.60209167){\color[rgb]{0,0,0}\makebox(0,0)[lt]{\lineheight{1.25}\smash{\begin{tabular}[t]{l}{\tiny $\gamma_1$}\end{tabular}}}}%
    \put(0.42326789,0.72456516){\color[rgb]{0,0,0}\makebox(0,0)[lt]{\lineheight{1.25}\smash{\begin{tabular}[t]{l}{\tiny $\gamma_2$}\end{tabular}}}}%
  \end{picture}%
\endgroup%
 
\caption{Inequivalent tunnels of $6_{11}$.}
\label{fig:two_tunnels_hk6_11}
\end{figure}

Let $\gamma_1,\gamma_2$ in Fig.\ \ref{fig:two_tunnels_hk6_11} 
be two tunnels of the handlebody knot $6_{11}$ in \cite{IshKisMorSuz:12}.
Table \ref{tab:tunnel_preserving_homo} lists
$\vert\mathcal{H}(M_p(\gamma_i))\vert$ for different $p$'s, and 
shows that $\gamma_i$, $i=1,2$, are inquivalent tunnels.

\begin{table}[h]
\caption{Homomorphisms to $A_5$.}
\begin{center}
\begin{tabular}{c|c|c}
    & $\vert\mathcal{H}(M_p(\gamma_1))\vert$ & $\vert\mathcal{H}(\mathcal{M}_p(\gamma_2))\vert$\\ 
    \hline
    $p=2$& $616$ & $615$\\
    \hline
    $p=3$& $528$ & $507$\\
    \hline   
    $p=5$& $676$&  $698$\\    
    \hline
\end{tabular}
\end{center}
\label{tab:tunnel_preserving_homo}
\end{table}

\end{document}